\documentclass[11pt,a4paper,reqno]{amsart}

\pdfoutput=1

\usepackage[utf8]{inputenc}
\usepackage[british]{babel}
\usepackage{amsmath,amssymb,amsthm}
\usepackage{thmtools,thm-restate}
\usepackage{hyperref,float,caption}
\usepackage[nameinlink]{cleveref}
\usepackage{geometry}
\usepackage{enumitem}
\usepackage[presets={vec-cev}]{letterswitharrows}
\usepackage[abbrev, msc-links, nobysame]{amsrefs}
\usepackage{tikz}
\usepackage{xcolor}
\usetikzlibrary{arrows}
\tikzset{
e+ /.tip = {[sep=0pt]|[sep=3pt]_},
e+4 /.tip = {[sep=0pt]|[sep=5pt]_},
e+5 /.tip = {[sep=0pt]|[sep=5pt]_},
e+6 /.tip = {[sep=0pt]|[sep=6pt]_},
e+7 /.tip = {[sep=0pt]|[sep=7pt]_},
e+2 /.tip = {[sep=0pt]|[sep=2pt]_}
}
\tikzset{arc/.style = {->,> = stealth'}}
\tikzset{
dot/.style = {circle, draw = black, fill, minimum size=#1,
              inner sep=0pt, outer sep=0pt},
dot/.default = 3.8pt
}
\usetikzlibrary{angles,quotes}
\newcommand{\x}[1]{\vecev{#1}}

\def\N{\mathbb N}

\def\P{\mathcal P}

\def\P{\mathcal P}
\def\X{\mathcal X}
\def\Y{\mathcal Y}

\def\halfedges{\textbf{E}}

\newcommand{\Bon}{\mathrm{Bon}}

\newcommand{\vehat}{\vec{\hat{e}}}

\newcommand{\set}[1]{\left\lbrace #1 \right\rbrace}
\newcommand{\inv}[1]{{#1}^{-}}

\newtheorem{theorem}{Theorem}[section]
\newtheorem{lemma}[theorem]{Lemma}
\newtheorem{corollary}[theorem]{Corollary}
\newtheorem{proposition}[theorem]{Proposition}

\theoremstyle{definition}
\newtheorem{definition}[theorem]{Definition}

\newtheorem{remark}[theorem]{Remark}
\newtheorem*{theorem*}{Theorem}

\crefname{enumi}{}{}
\crefname{enumii}{}{}
\crefformat{enumi}{#2#1#3}
\crefformat{enumii}{#2#1#3}
\Crefformat{enumi}{#2#1#3}
\Crefformat{enumii}{#2#1#3}

\crefname{definition}{definition}{definitions}
\crefformat{definition}{#2Definition~#1#3}
\Crefformat{definition}{#2Definition~#1#3}

\crefname{section}{section}{sections}
\crefformat{section}{#2Section~#1#3}
\Crefformat{section}{#2Section~#1#3}

\crefname{subsection}{Subsection}{subsections}
\crefformat{subsection}{#2Subsection~#1#3}
\Crefformat{subsection}{#2Subsection~#1#3}

\crefname{lemma}{lemma}{lemmata}
\crefformat{lemma}{#2Lemma~#1#3}
\Crefformat{lemma}{#2Lemma~#1#3}

\crefname{remark}{remark}{remarks}
\crefformat{remark}{#2Remark~#1#3}
\Crefformat{remark}{#2Remark~#1#3}

\crefname{theorem}{theorem}{theorems}
\crefformat{theorem}{#2Theorem~#1#3}
\Crefformat{theorem}{#2Theorem~#1#3}

\crefname{corollary}{corollary}{corollaries}
\crefformat{corollary}{#2Corollary~#1#3}
\Crefformat{corollary}{#2Corollary~#1#3}

\crefname{figure}{figure}{figures}
\crefformat{figure}{#2Figure~#1#3}
\Crefformat{figure}{#2Figure~#1#3}

\crefname{proposition}{proposition}{propositions}
\crefformat{proposition}{#2Proposition~#1#3}
\Crefformat{proposition}{#2Proposition~#1#3}

\crefname{observation}{observation}{observations}
\crefformat{observation}{#2Observation~#1#3}
\Crefformat{observation}{#2Observation~#1#3}

\title{Menger's Theorem in bidirected graphs}

\author[N.\,Bowler]{Nathan Bowler}
\author[E.\,Ghorbani]{Ebrahim Ghorbani}
\address{Ebrahim Ghorbani:
Department of Mathematics, K.N. Toosi University of Technology, P. O. Box 16765-3381, Tehran, Iran and Universit{\"a}t Hamburg, Department of Mathematics, Bundesstra{\ss}e~55 (Geomatikum), 20146~Hamburg, Germany}
\email{ghorbani@kntu.ac.ir, ebrahim.ghorbani@uni-hamburg.de}
\author[F.\,Gut]{Florian Gut}
\author[R.\,W.\,Jacobs]{Raphael W. Jacobs}
\author[F.\,Reich]{Florian Reich}
\address{Nathan Bowler, Florian Gut, Raphael W. Jacobs, Florian Reich: Universit{\"a}t Hamburg, Department of Mathematics, Bundesstra{\ss}e~55 (Geomatikum), 20146~Hamburg, Germany}
\email{\{nathan.bowler,  florian.gut, raphael.jacobs, florian.reich\}@uni-hamburg.de}

\begin{document}

\begin{abstract}
    Bidirected graphs are a generalisation of directed graphs that arises in the study of undirected graphs with perfect matchings.
    Menger's famous theorem -- the minimum size of a set separating two vertex sets $X$ and $Y$ is the same as the maximum number of disjoint paths connecting them -- is generally not true in bidirected graphs.
    We introduce a sufficient condition for~$X$ and~$Y$ which yields a version of Menger's Theorem in bidirected graphs that in particular implies its directed counterpart.
\end{abstract}

\maketitle

\section{Introduction}

A fundamental result in the field of graph theory is due to Karl Menger~\cite{Menger1927} which is nowadays just known as Menger's Theorem.
It gives an important insight into the connectivity of two sets of vertices in directed graphs.
In its common form, Menger's Theorem asserts that the minimum number of vertices separating two sets of vertices $X$ and $Y$ in a directed graph $D$ is the same as the maximum number of disjoint $X$--$Y$ paths in $D$, see e.g.\ \cite{bang2008}*{Theorem~7.3.1}.

A slight structural strengthening of Menger's Theorem was proven by B{\"o}hme, G{\"o}ring and Harant~\cite{Boehme01}:
For a set $\P$ of $k$ disjoint $X$--$Y$ paths, there is either a set of $k$ vertices separating $X$ from $Y$, one vertex on each path in~$\P$, or a set of $k+1$ disjoint $X$--$Y$ paths where $k$ of them use the same startvertices in $X$ as the paths of $\P$:

\begin{theorem}[\cite{Boehme01}*{Theorem~2}] \label{theo:DirectedMenger}
    Let~$X$ and~$Y$ be sets of vertices of a directed graph~$D$, and let~$P_1, \dots, P_k$ be vertex-disjoint $X$--$Y$~paths in $D$ where~$P_i$ starts in~$v_i \in X$ for~$i \in [k]$.
    Then precisely one of the following is true:
    \begin{enumerate}[label=(\arabic*)]
        \item There is a set $S$ of $k$ vertices of $D$ such that $D - S$ contains no $X$--$Y$~path.
        \item There are~$k + 1$ vertex-disjoint $X$--$Y$~paths $P_1', \dots, P_{k+1}'$ in~$D$ where~$P_i'$ starts in~$v_i$ for~$i \in [k]$.
    \end{enumerate}
\end{theorem}

In this paper we make this stronger form of Menger's Theorem available in the realm of bidirected graphs.
Bidirected graphs were first introduced by Kotzig in a series of papers~\cites{Kotzig1959,Kotzig1959b,Kotzig1960}.
These objects can be understood as a generalisation of undirected and directed graphs.
They can be obtained from undirected graphs by assigning to each endpoint of every edge one of two signs, $+$ or $-$.
Directed graphs can then be envisioned as bidirected graphs in which each edge has the sign~$+$ at one endpoint and~$-$ at the other.

Bidirected graphs have recently received increased attention by the graph theoretic community with both structural~\cites{Appa2006,wiederrecht2020thesis} and algorithmic results~\cites{Ando1996,Bolker2006}.
Especially, Wiederrecht proposed bidirected graphs as another angle of attack to investigate the structure of undirected graphs with perfect matchings.
This approach is inspired by a recent proof of Norin's Matching Grid Conjecture for bipartite graphs with perfect matchings, which was given by Hatzel, Rabinovich and Wiederrecht~\cite{hatzel2019cyclewidth}.
Their proof applies the Directed Grid Theorem~\cite{kawarabayashi2015directed} using the structural relation of directed graphs and bipartite graphs with perfect matchings.
In the same structural sense, bidirected graphs are in a one-to-one correspondence with general undirected graphs with perfect matchings; for a detailed introduction, we refer the reader to~\cite{wiederrecht2020thesis}.

Unfortunately, Menger's Theorem as in~\cref{theo:DirectedMenger} does not hold if we simply replace `directed' by `bidirected' as illustrated in \cref{theo:counterexample_to_bidirected_menger} and \cref{fig:counterexample_vertex_menger}.
The main reason for this is an intricate complication in the structure of bidirected graphs: unlike in directed graphs, a walk between two fixed vertices need not contain a path between them.
This property of bidirected graphs in particular prevents a direct transfer of the usual proofs of~\cref{theo:DirectedMenger} such as the one given in~\cite{Boehme01}.

To overcome this complication, we introduce a sufficient condition on the set~$\X$ of `signed' startvertices of our paths, where a \emph{signed vertex} is a pair of a vertex and a sign~$+$ or~$-$: we require~$\X$ to be `clean', that is, we forbid the existence of nontrivial~$\X$--$\X$ paths.
This condition is satisfied in particular by any set of startvertices in a directed graph.
Our main result then reads as follows:

\begin{restatable}{theorem}{VertexMenger}\label{theo:vertex_menger}
    Let~$\X$ and~$\Y$ be sets of signed vertices of a bidirected graph~$B$, and suppose that $\X$ is clean.
    Let~$P_1, \dots, P_k$ be vertex-disjoint $\X$--$\Y$~paths in $B$ where~$P_i$ starts in~$v_i \in V(\X)$ for~$i \in [k]$.
    Then precisely one of the following is true:
    \begin{enumerate}[label=(\arabic*)]
        \item There is a set $S$ of $k$ vertices of $B$ such that $B - S$ contains no $\X$--$\Y$~path. \label{item:vtx-sep}
        \item There are~$k + 1$ vertex-disjoint $\X$--$\Y$~paths $P_1', \dots, P_{k+1}'$ in~$B$ where~$P_i'$ starts in~$v_i$ for~$i \in [k]$. \label{item:bidirected-vtx-paths}
    \end{enumerate}
\end{restatable}

The general strategy of our proof of \cref{theo:vertex_menger} follows the idea of \cite{Boehme01}, but needs several extensions in order to circumnavigate the above mentioned complication in bidirected graphs.
In particular, we introduce `appendages' as our main tool to overcome the walk-path problem for clean sets of vertices.
This allows us to first prove an edge-version of Menger's Theorem, \cref{theo:edge_menger}, from which we then deduce the vertex-version, \cref{theo:vertex_menger}.
Our proof of~\cref{theo:vertex_menger} can also be turned into a polynomial-time algorithm which finds a maximal sets of disjoint paths.

\textbf{Structure of the paper:}
After introducing notation and general concepts around bidirected graphs in \cref{sec:Preliminaries}, we show in \cref{sec:counterexample} that Menger's Theorem cannot be transferred directly to bidirected graphs.
We then introduce `appendages' in \cref{sec:appendages} and use it to first prove Menger's Theorem for edge-disjoint paths in~\cref{sec:edges-disjoint}.
Finally, we deduce the vertex-version, \cref{theo:vertex_menger}, from the edge-disjoint one in \cref{sec:vertex-disjoint}.
In~\cref{sec:polynomial_time_algorithm} we describe how our proof of~\cref{theo:vertex_menger} can be turned into a polynomial-time algorithm.

\section{Preliminaries}\label{sec:Preliminaries}

Throughout the paper we assume $\N = \{0, 1, 2, \dots\}$.
For numbers~$k \leq \ell$ we write~$[k, \ell] := \{n \in \N : k \leq n \leq \ell\}$, and we abbreviate~$[k] := [1, k]$.

\subsection{Basic Definitions for Bidirected Graphs}

For the basic definitions concerning undirected graphs we refer the reader to \cite{DiestelBook2016} and for the notions around directed graphs we refer the reader to \cite{bang2008}.
Undirected and directed graphs considered in this paper are allowed to have parallel edges, but they do not have loops.

A \emph{bidirected graph}~$B = (G, \sigma)$ consists of a graph~$G = (V, E)$, a corresponding set of \emph{half-edges} defined as
    \begin{equation*}
        \halfedges(B) := \set{ (u, e) \colon e \in E \text{ and } u \in e }
    \end{equation*}
and a \emph{signing}~$\sigma: \halfedges \to \set{+, -}$ assigning to each half-edge~$(u, e)$ its \emph{sign}~$\sigma(u, e) := \sigma((u,e))$; we say that~\emph{$e$ has sign~$\sigma(u,e)$ at~$u$}.
Then~$V(B) := V$ is the \emph{vertex set} of~$B$ and~$E(B) := E$ is its \emph{edge set}.
We refer to the elements of~$V(B)$ and~$E(B)$ as the \emph{vertices} and the \emph{edges} of~$B$, respectively.

For technical simplification, we do not allow bidirected graphs to have distinct edges~$e$ and~$f$ that have the same endvertices and the same signs at them.
This simplification does not affect the main theorem, \cref{theo:vertex_menger}, which still holds without this assumption.
Indeed, one can easily deduce the statement for bidirected graphs with parallel edges by applying \cref{theo:vertex_menger} to the bidirected graph obtained by subdividing each edge and giving the two edges incident to a newly arising subdivision vertex distinct signs at it.

A \emph{signed vertex} of~$B$ is a pair~$(v, \alpha)$ of a vertex~$v$ of~$B$ and a sign~$\alpha \in \{+, -\}$.
Given a set~$\X$ of signed vertices, we write~$V(\X) := \{ v \in V(B) \mid \exists\ \alpha \in \{+, -\}: (v, \alpha) \in \X \}$.

An \emph{oriented edge}~$\ve$ of a bidirected graph~$B$ is formally defined as a triple~$(e, u, v)$ where~$e$ is an edge of~$B$ with incident vertices~$u, v \in V(B)$; we call~$e$ its \emph{underlying edge},~$u$ its \emph{startvertex} and~$v$ its~\emph{endvertex}, and think of~$\ve$ as orienting~$e$ from~$u$ to~$v$.
The edge~$e$ has precisely two \emph{orientations}, one with startvertex~$u$ and endvertex~$v$ and the other one with startvertex~$v$ and endvertex~$u$.
We denote the two orientations of~$e$ as~$\ve$ and~$\ev$; there is no default orientation of~$e$, but if we are given one of them as~$\ve$, say, then the other one is written as~$\ev$.
Given an oriented edge~$\ve$ we conversely write~$e$ for its underlying edge.
Given a set~$A$ of edges of~$B$, we write~$\vAv$  for the set of all orientations of edges in~$A$.

In every \namecref{fig:counterexample_edge_menger} involving bidirected graphs, we depict the signs of halfedges in a particular manner by drawing the signs onto the edge, which results in a bar perpendicular to the edge at incident vertices with sign $+$ and none at incident vertices with sign $-$, see e.g.\ \cref{fig:counterexample_edge_menger}.

\subsection{Connectivity in Bidirected Graphs}

Let~$B = (G, \sigma)$ be a bidirected graph.
A sequence
\begin{equation*}
    W = v_0 \ve_1 v_1 \ve_2 v_2 \dots v_{\ell - 1} \ve_\ell v_\ell
\end{equation*}
of vertices~$v_i \in V(B)$ and oriented edges~$\ve_j \in \vE(B)$ is a~\emph{walk}~$W$ of \emph{length}~$\ell$ in~$B$ if the oriented edge~$\ve_j$ has startvertex~$v_{j-1}$ and endvertex~$v_j$ for~$j \in [\ell]$ and we have~$\sigma(v_i, e_i) \neq \sigma(v_i, e_{i+1})$ for~$i \in [\ell - 1]$.
A \emph{subwalk} of~$W$ is a walk of the form~$v_i \ve_i \dots \ve_{j-1} v_{j}$ for some~$i \le j$.
The \emph{inverse of~$W$} is~$\inv{W} := v_\ell \ev_{\ell} v_{\ell-1} \ev_{\ell-1} v_{\ell-2} \dots v_1 \ev_1 v_0$.
A walk is \emph{trivial} if it has length~$0$ and \emph{nontrivial} otherwise.

The walk $W$ \emph{starts} in its~\emph{startvertex}~$v_0$ and \emph{ends} in its \emph{endvertex}~$v_\ell$.
If $W$ is nontrivial, we say that it starts with sign~$\sigma(e_1, v_0)$ in its \emph{signed startvertex} $(v_0, \sigma(e_1, v_0))$.
Likewise we say that it ends with sign~$\sigma(e_\ell, v_\ell)$ in its \emph{signed endvertex} $(v_\ell, \sigma(e_\ell, v_\ell))$.
All other vertices of~$W$ are \emph{internal vertices} of~$W$.
The set of all vertices in~$W$ is denoted as~$V(W) := \{v_0, \dots, v_\ell\}$.
The nontrivial walk~$W$ \emph{arrives in~$v_i$} with sign~$\sigma(e_i, v_i)$ for~$i \in [1, \ell]$ and \emph{leaves~$v_i$} with sign~$\sigma(e_{i+1}, v_i)$ for~$i \in [0, \ell - 1]$.
Analogously, we say that a nontrivial walk $W$ \emph{starts} in~$\ve_1$ and \emph{ends} in~$\ve_\ell$.

The set of all oriented edges in a walk~$W$ is denoted by~$\vE(W) := \{\ve_1, \dots, \ve_\ell\}$, the set of edges underlying~$\vE(W)$ is~$E(W) := \{e_1, \dots, e_\ell\}$, and we write~$\vEv(W) := \vecev{E(W)}$.
If~$E(W) \subseteq A$ for some set~$A$ of edges of~$B$, then~$W$ is a walk \emph{in~$A$}.

A nontrivial walk~$W$ is an~\emph{$x$--$y$~walk} in~$B$ if~$W$ starts in~$x$ and ends in~$y$, where~$x$ and~$y$ can be oriented edges or vertices or signed vertices of~$B$.
Given two sets~$\X$ and~$\Y$ of signed vertices of~$B$ we define an~\emph{$\X$--$\Y$~walk} as follows:
a trivial walk $W = v$ is an $\X$--$\Y$ walk if there is $\alpha \in \{ +, - \}$ such that $(v,\alpha) \in \X$ and $(v,-\alpha) \in \Y$.
A nontrivial walk $W$ is an $\X$--$\Y$ walk if it is an $x$--$y$ walk for some $x \in \X$ and $y \in \Y$, no internal vertex of $W$ is in $V(\X) \cup V(\Y)$ and neither the start- nor the endvertex of $W$ forms a trivial $\X$--$\Y$ walk.
Note that no proper subwalk of an $\X$--$\Y$~walk is again an $\X$--$\Y$~walk.

Two walks~$W$ and~$W'$ in~$B$ are \emph{vertex-disjoint} if~$V(W) \cap V(W') = \emptyset$.
Similarly,~$W$ and~$W'$ are \emph{edge-disjoint} if~$E(W) \cap E(W') = \emptyset$.
A set of walks is \emph{vertex-disjoint} (respectively \emph{edge-disjoint}) if its elements are pairwise vertex-disjoint (respectively edge-disjoint).\\

A walk~$W$ in~$B$ is a~\emph{trail} if all edges underlying~$\vE(W)$ are distinct.
If not only all edges, but also all vertices of~$W$ are distinct, then~$W$ is a \emph{path}.
For both trails and paths we transfer all notions and notation defined for walks (such as e.g.\ subtrail or~$v$--$w$~path).

A key difference between bidirected and both undirected and directed graphs which complicates understanding their connectivity properties considerably lies in the relation of walks, trails and paths:
unlike for (un)directed graphs, the existence of a~$v$--$w$~walk between two vertices~$v$ and~$w$ in a bidirected graph~$B$ does not imply that there exists a~$v$--$w$~trail in~$B$, and similarly a~$v$--$w$~trail does not guarantee the existence of a~$v$--$w$~path (see~\cite{wiederrecht2020thesis}*{Figure 9.1} for examples). \\

For a given trail~$Q = v_0 \ve_1 v_1 \ve_2 v_2 \dots v_{\ell - 1} \ve_\ell v_\ell$ and~$0 \le i \le j \le \ell$, we write $\ve_i Q := v_{i-1} \ve_i \dots v_\ell$, $Q \ve_i := v_0 \dots \ve_i v_i$ and~$\ve_i Q \ve_j := v_{i-1} \ve_i \dots \ve_j v_j$ for the appropriate subtrails of~$Q$.
Analogously, given a path~$P = v_0 \ve_1 v_1 \ve_2 v_2 \dots v_{\ell - 1} \ve_\ell v_\ell$ and~$0 \le i \le j \le \ell$, we write~$v_i P := v_i \dots v_\ell$,~$P v_i := v_0 \dots v_i$ and~$v_i P v_j := v_i \dots v_j$ for the appropriate subpaths of~$P$, and let~$\ve_i P v_j := v_{i-1} P v_j$ as well as~$v_i P \ve_j := v_i P v_j$.
In particular, these two notions combine for paths to the respective notations~$\ve_i P v_j$ and~$v_i P \ve_j$.

We use an analogous notion for the concatenation of trails at edges:
for example if the union~$Q_1 \ve \cup \ve Q_2 \vf \cup \vf Q_3$ of three trails is again a trail, we denote it as~$Q_1 \ve Q_2 \vf Q_3$.
We write similarly~$P_1 x P_2 y P_3$ for paths where~$x, y$ may be both vertices or oriented edges. \\

The complicated connectivity structure of bidirected graphs also manifests in various different notions of strong connectivity (see~\cite{wiederrecht2020thesis}*{Section~9.2} for an overview).
We discuss two of them, namely `strongly connected' and `circularly connected', in the remainder of this section.
Let us first define `strongly connected':

\begin{definition}
    We say a bidirected graph $B$ is {\em strongly connected} if for any two vertices $v$ and $w$ of $B$ there are signs $\alpha$ and $\beta$ such that $B$ contains both a $(v,\alpha)$--$(w,\beta)$ path and a $(v, -\alpha)$--$(w, -\beta)$ path.
\end{definition}

\noindent It turns out that an apparently weaker condition is equivalent to the one in this definition:

\begin{lemma}\label{lem:weakerstrong}
    Let $B$ be a bidirected graph such that for any two vertices $v$ and $w$ of $B$ there are a $(v,+)$--$w$ path and a $(v,-)$--$w$ path in $B$.
    Then $B$ is strongly connected.
\end{lemma}
\begin{proof}
    Let $v$ and $w$ be vertices of $B$. We must find paths joining them as in the definition of strong connectivity. By assumption there are signs $\beta_1$ and $\beta_2$ such that $B$ contains a $(v, +)$--$(w, \beta_1)$ path and a $(v, -)$--$(w, \beta_2)$ path. Similarly, applying the assumption to $w$ and $v$ then reversing the paths shows that there are signs $\alpha_1$ and $\alpha_2$ such that $B$ contains a $(v, \alpha_1)$--$(w, +)$ path and a $(v, \alpha_2)$--$(w, -)$ path. 
    
    If $\beta_2 = - \beta_1$ then we are done, so we may suppose $\beta_1 = \beta_2$. Similarly if $\alpha_2 = -\alpha_1$ then we are done, so we may suppose $\alpha_1 = \alpha_2$. Setting $\alpha := -\alpha_1$ and $\beta := \beta_1$ we are done.
\end{proof}

We now turn to `circularly connected'.
For this, a \emph{cycle} in a bidirected graph~$B = (G, \sigma)$ is a trail~$C = v_0 \ve_1 v_1 \ve_2 v_2 \dots v_{\ell - 1} \ve_\ell v_\ell$ in~$B$ whose vertices are all distinct except~$v_0 = v_\ell$ where we also have~$\sigma(v_0, e_1) \neq \sigma(v_\ell, e_\ell)$.

\begin{definition}
    Let $B = (G,\sigma)$ be a bidirected graph and consider the undirected graph $H := (V(G),F)$ where $F$ is the set of edges of $G$ that lie on some cycle of $B$.
    We refer to the connected components of $H$ as the {\em circular components} of $B$.
    If $H$ is connected we call $B$ \emph{circularly connected}.
\end{definition}

We now aim to show that `strongly connected' and `circularly connected' in fact describe the same type of connectivity.
This result was already known (see for example \cite{wiederrecht2020thesis}*{Section~9.2}), but we include a proof here for the convenience of the reader.

\begin{theorem}\label{theo:strongly_connected_equivalent_to_circularly_connected}
    A bidirected graph $B = (G,\sigma)$ is strongly connected if and only if it is circularly connected.
\end{theorem}

\noindent For the proof of~\cref{theo:strongly_connected_equivalent_to_circularly_connected}, we make use of a lemma that describes which types of connectivity are forced by the interaction of paths between two vertices that differ in the signs at their endpoints:

\begin{lemma}\label{lem:newsigns}
    Let $v$ and $w$ be vertices of a bidirected graph $B$.
    Let $\alpha$ and $\beta$ be signs such that $B$ contains both a $(v,\alpha)$--$(w,\beta)$ path $P$ and a $(v, -\alpha)$--$(w, -\beta)$ path $Q$. Then at least one of the following statements holds:
    \begin{itemize}
        \item $v$ and $w$ lie in the same circular component of $B$.
        \item There is a $(v, \alpha)$--$(w, -\beta)$ path in $B$.
    \end{itemize}
\end{lemma}
\begin{proof}
    The proof proceeds by induction on the sum of the lengths of the paths $P$ and $Q$.
    If this sum is $0$, then $v = w$ and therefore the first statement holds.
    Otherwise, let $x$ be the first vertex of $P$ other than $v$ which also lies on $Q$. Suppose that $xP$ is an $(x, \gamma)$--$(w, \beta)$ path and that $xQ$ is an $(x,\delta)$--$(w, -\beta)$ path. 
    \begin{description}
        \item[Case 1: $\gamma = \delta$] In this case, $PxQ$ is a $(v, \alpha)$--$(w, -\beta)$ path in $B$.
        \item[Case 2: $\gamma = -\delta$] In this case we can apply the induction hypothesis to $x$, $w$, $\gamma$, $\beta$ and the paths $xP$ and $xQ$ in the bidirected graph $B'$ given by the union of $xP$ and $xQ$. There are two possibilities.
        \begin{description}
            \item[Case 2.1: $x$ and $w$ lie in the same circular component of $B$] As the path $Px$ is a $(v, \alpha)$--$(x, -\gamma)$~path, $Qx$ is a $(v, -\alpha)$--$(x, \gamma)$~path, and these paths are disjoint except at their endvertices, their union is a cycle. This cycle witnesses that $v$ is in the same circular component as $x$, hence also as $w$.
            \item[Case 2.2: There is an $(x, \gamma)$--$(w, -\beta)$ path $R$ in $B'$] In this case, $PxR$ is a $(v, \alpha)$--$(w, -\beta)$ path in $B$. 
        \end{description}
    \end{description}
\end{proof}

\begin{proof}[Proof of~\cref{theo:strongly_connected_equivalent_to_circularly_connected}]
    Suppose first that $B$ is circularly connected.
    By \cref{lem:weakerstrong} it is enough to show that for any vertices $v$ and $w$ of $B$ there are both a $(v, +)$--$w$ path and a $(v, -)$--$w$ path.
    Let $v \in V(B)$ and $\alpha \in \{+, - \}$ be arbitrary.
    We set $X \subseteq V(B)$ to be the set of all vertices $x$ for which there is a $(v, \alpha)$--$x$ path.
    Note that the set $X$ is non-empty since it contains $v$.
    We show that $X$ is the whole vertex set of $B$.
    Suppose not, then by circular connectivity there is some edge $e$ of $B$ contained in some cycle $C$ of $B$ which joins some~$x \in X$ and some~$y \not \in X$.

    There is a $(v, \alpha)$--$x$ path $P$ by definition of $X$.
    Let $z$ be the first vertex of $P$ on $C$ and let $\beta \in \{+, -\}$ such that the (possibly trivial) path $Pz$ ends in $(z, \beta)$.
    Then there is a $(z, -\beta)$--$y$ path $Q$ contained in $C$.
    Thus, $PzQ$ is a $(v, \alpha)$--$y$ path, contradicting $y \not \in X$.
    Since this is true for any $v \in V(B)$ and any $\alpha \in \{+,-\}$, this completes this direction of the proof.
    
    Now suppose instead that $B$ is strongly connected, and suppose for a contradiction that it is not circularly connected.
    Then it must contain an edge $e$ which joins two vertices $x$ and $y$ in different circular components.
    Let $\sigma(x,e) = -\alpha$ and $\sigma(y,e) = \beta$.
    Suppose there is no $(x,\alpha)$--$(y,-\beta)$ path.
    Since $B$ is strongly connected, there must be an $(x,\alpha)$--$(y,\beta)$ path and an $(x,-\alpha)$--$(y,-\beta)$ path.
    Applying \cref{lem:newsigns}, we yield an $(x,\alpha)$--$(y,-\beta)$ path in $B$.
    That path together with $e$ is a cycle in $B$, contradicting our assumption that $x$ and $y$ lie in different circular components.
\end{proof}

\section{Counterexample for General Menger Theorem}\label{sec:counterexample}
In this section we show that Menger's Theorem, \cref{theo:DirectedMenger}, does not hold true in bidirected graphs if we transfer the statement verbatim. Even if we allow the separating set $S$ to have any fixed size $k$ for some $k \in \N$, the statement is false:

\begin{theorem}\label{theo:counterexample_to_bidirected_menger}
    For each number~$k \in \N$, there exists a bidirected graph~$B$ and disjoint sets~$\X$ and~$\Y$ of signed vertices of~$B$, each of size at least~$k$, such that there are no two disjoint~$\X$--$\Y$~paths and for each subset~$S \subseteq V(B)$ of size~$k$ there exists an~$\X$--$\Y$~path in~$B - S$.
\end{theorem}

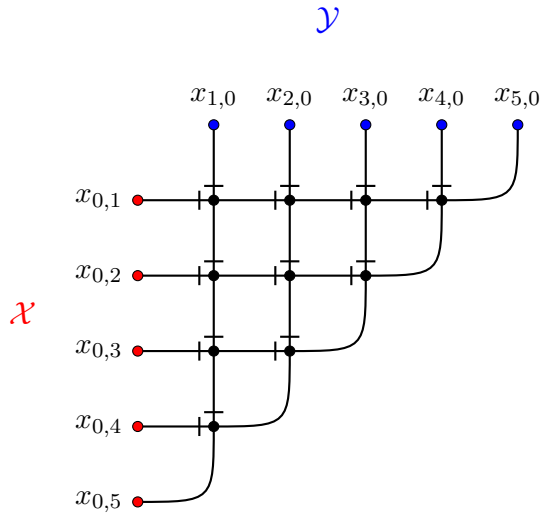
\begin{figure}[h]
\centering
	\begin{tikzpicture}
\foreach \i in {1,...,5}{
    \draw [thick] (\i-1,\i-6) ..controls+(right:1cm) and +(down:1cm)  .. (\i,\i-5);
    \foreach \j in {1,...,5}
        {\ifnum\numexpr \i+\j<6		
        \node[dot] at (\i,-\j) [] {};
        \draw[-e+4, thick]  (\i,-\j+1) to (\i,-\j);
        \draw[-e+4, thick] (\i-1,-\j) to (\i,-\j);
        \fi}
        \node[dot,fill = red] at (0,-\i) [label=left:$x_{0,\i}$]{};	
    \node[dot,fill = blue]  at (\i,0)  [label=above:$x_{\i,0}$]{};
	}
\node[left, red] at (-1.2,-2.5) {\large$\X$};
\node[above, blue] at (2.5,1.1) {\large$\Y$};
\end{tikzpicture}
\caption{A bidirected graph containing neither two vertex-disjoint \textcolor{red}{$\X$}--\textcolor{blue}{$\Y$} paths nor a vertex set $W$ of size $\leq 2$ such that in $B - W$ there are no \textcolor{red}{$\X$}--\textcolor{blue}{$\Y$} paths.}
    \label{fig:counterexample_vertex_menger}
\end{figure}

\noindent For the proof of \cref{theo:counterexample_to_bidirected_menger} we will rely on the following well-known topological lemma, which follows directly from the fact that a complete graph on five vertices is not planar:

\begin{lemma}\label{lem:topology}
    Let $x_1$, $x_2$, $y_1$ and $y_2$ be points appearing in the clockwise order on the boundary of the closed disc. Then any arc from $x_1$ to $y_1$ meets any arc from $x_2$ to $y_2$. \qed
\end{lemma}

\begin{proof}[Proof of \cref{theo:counterexample_to_bidirected_menger}]
    The vertices of $B$ will be given by vertices $x_{i,j}$ indexed by pairs $(i,j)$ of natural numbers with $1 \leq i + j \leq 2k + 1$.
    We add edges of three kinds.
    For any natural numbers $i$ and $j$ with $1 \leq i$ and $i + j \leq 2k$ we add an edge that is incident to $x_{i,j}$ with sign~$-$ and incident to $x_{i,j+1}$ with sign~$+$.
    Similarly for any natural numbers $i$ and $j$ with $1 \leq j$ and $i + j \leq 2k$ we add an edge that is incident to $x_{i,j}$ with sign~$-$ and incident to $x_{i+1,j}$ with sign~$+$.
    Finally, for any $i \leq 2k$ we add an edge incident to $x_{i,2k+1-i}$ and $x_{i+1, 2k-i}$, with sign~$-$ at both vertices.
    We call the edges of these three kinds {\em vertical}, {\em horizontal} and {\em diagonal edges} respectively -- see \cref{fig:counterexample_vertex_menger} for an illustration where $k=2$.
    We set $\X := \{x_{0,j} \colon 1 \leq j \leq 2k+1\}$ and $\Y := \{x_{i,0} \colon 1 \leq i \leq 2k+1\}$. 
    
    We show that for any set $S \subseteq V(B)$ of size $k$ there is an $\X$--$\Y$~path in $B - S$.
    First we construct a sequence of $\X$--$\Y$ paths $P_1, P_2, \ldots, P_{2k+1}$ by taking the path $P_i$ to be the concatenation of the path from $x_{0,i}$ to $x_{2k+1-i, i}$ consisting of horizontal edges, followed by the diagonal edge from $x_{2k+1-i, i}$ to $x_{2k+2-i,i-1}$ and then the path consisting of vertical edges to $x_{2k+2-i, 0}$.
    Note that no vertex appears in more than two of these paths, so the number of these paths meeting $S$ is at most $2k$.
    Since there are $2k+1$ of these paths, one of them avoids~$S$.
    
    Now suppose for a contradiction that there are two disjoint $\X$--$\Y$ paths $P$ and $Q$.
    Let the initial vertices of $P$ and $Q$ be $x_P$ and $x_Q$ and let their final vertices be $y_P$ and $y_Q$.
    Since deleting all the diagonal edges from $B$ leaves a directed graph in which no directed edge points into $\X$ or $\Y$, both $P$ and $Q$ must contain diagonal edges.
    Let $z_P$ and $z_Q$ be endvertices of those diagonal edges on $P$ and $Q$ respectively which are distinct from $x_P$, $x_Q$, $y_P$ and $y_Q$.
    We may embed $B$ in the disc in such a way that on the boundary we have, in clockwise order, first the elements of $\X$, then those of $\Y$, then all other $x_{i,j}$ with $i+j=2k+1$.
    Without loss of generality we may assume that in the clockwise order after $\Y$ on the boundary of the disc we first reach $z_Q$ then $z_P$.
    Then disjointness of the arcs induced by the paths $z_PPy_P$ and $x_QQz_Q$ contradicts \cref{lem:topology}, completing the proof.
\end{proof}

We will show in \cref{sec:vertex-disjoint} that the vertex-version of Menger's Theorem in the context of bidirected graphs follows from the edge-version, meaning that the counterexample above also implies the existence of a counterexample to the edge-version of Menger's Theorem in our context.
It is also not difficult to construct counterexamples to the vertex-version directly with a slight modification of the construction above, as illustrated in \cref{fig:counterexample_edge_menger}.

\begin{theorem}\label{thm:counterexample_to_edge_menger}
    For each number~$k \in \N$, there exists a bidirected graph~$B$ and distinct vertices ~$x$ and~$y$, such that there are no two edge-disjoint $x$--$y$~paths and for each subset $S \subseteq E(B)$ of size $k$ there exists an $x$--$y$~path in $B - S$.
    \hfill$\square$
\end{theorem}

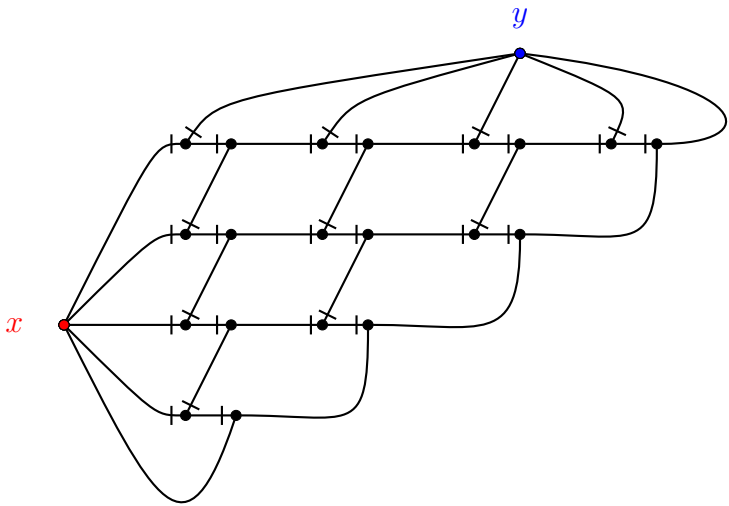
\begin{figure}[ht]
    \centering
	\begin{tikzpicture}[scale=.8]
	\node[dot, fill = red] (a0) at (-0.5,3) {};
	\node[dot, fill = red] (b0) at (-0.5,3) {};
	\node[dot] (b2) at (1.5,1.5) {};
	\node[dot] (b3) at (2.33,1.5) {};
	\node[dot, fill = red] (c0) at (-0.5,3) {};
	\node[dot] (c2) at (1.5,3) {};
	\node[dot] (c3) at (2.25,3) {};
	\node[dot] (c5) at (3.75,3) {};
	\node[dot] (c6) at (4.5,3) {};
	\node[dot, fill = red] (d0) at (-0.5,3) {};
	\node[dot] (d2) at (1.5,4.5) {};
	\node[dot] (d3) at (2.25,4.5) {};
	\node[dot] (d5) at (3.75,4.5) {};
	\node[dot] (d6) at (4.5,4.5) {};
	\node[dot] (d8) at (6.25,4.5) {};
	\node[dot] (d9) at (7,4.5) {};
	\node[dot, fill = red] (e0) at (-0.5,3) {};
	\node[dot] (e2) at (1.5,6) {};
        \node[dot] (e3) at (2.25,6) {};
        \node[dot] (e5) at (3.75,6) {};
        \node[dot] (e6) at (4.5,6) {};
        \node[dot] (e8) at (6.25,6) {};
        \node[dot] (e9) at (7,6) {};
	\node[dot] (e11) at (8.5,6) {};
	\node[dot] (e12) at (9.25,6) {};
	\node[dot, fill = blue] (v2) at (7,7.5) {};
	\node[dot, fill = blue] (v5) at (7,7.5) {};
	\node[dot, fill = blue] (v8) at (7,7.5) {};
	\node[dot, fill = blue] (v11) at (7,7.5) {};
	\node[dot, fill = blue] (v13) at (7,7.5) {};

    \draw[-e+, thick] (b0) ..controls (1,1.5) .. (b2);
    \draw[-e+, thick] (c0) ..controls (1,3) .. (c2);
    \draw[-e+, thick] (d0) ..controls (1,4.5) .. (d2);
    \draw[-e+, thick] (e0) ..controls (1,6) .. (e2);

	\foreach \x in {b,c,d,e} 
	 {\draw[-e+, thick] (\x2) to (\x3);}
\foreach \x in {c,d,e} 
{\draw[-e+2, thick] (\x5) to (\x6);}
\foreach \x in {d,e} 
{\draw[-e+2, thick] (\x8) to (\x9);}
\draw[-e+2, thick] (e11) to (e12);

\draw[-e+, thick] (v2) ..controls(2,6.75) .. (e2);
\draw[-e+, thick] (v5) ..controls(4.25,6.75) .. (e5);
\draw[-e+, thick] (v8) to (e8);
\draw[-e+, thick] (v11) ..controls(8.85,6.75) .. (e11);

\foreach \x/\y in {b3/c6,c6/d9,d9/e12} 	
\draw [thick] (\x) ..controls+(right:1.9cm) and +(down:1.9cm)  .. (\y);
\draw [thick] (a0) ..controls (1, 0) and (1.5,-1.0)  .. (b3);
\draw [thick] (e12) ..controls(11, 6) and (11, 7.0)  .. (v13);
\foreach \x/\y in {2/3,5/6,8/9}
\draw[-e+2, thick] (e\y) to (d\x);
\foreach \x/\y in {2/3,5/6}
\draw[-e+2, thick] (d\y) to (c\x); 	
\draw[-e+2, thick] (c3) to (b2); 	
\foreach \x in {c,d,e}
\draw[-e+2, thick] (\x3) to (\x5); 
\foreach \x in {d,e}
\draw[-e+2, thick] (\x6) to (\x8); 
\draw[-e+2, thick] (e9) to (e11); 
\node[left, red] at (-1,3) {\large$x$};
\node[above, blue] at (7,7.75) {\large$y$};
\end{tikzpicture}
\caption{A bidirected graph containing neither two edge-disjoint \textcolor{red}{$x$}--\textcolor{blue}{$y$} paths nor an edge set $W$ of size $\leq 2$ such that in $B - W$ there is no \textcolor{red}{$x$}--\textcolor{blue}{$y$} path.}
    \label{fig:counterexample_edge_menger}
\end{figure}

\section{Appendages of paths}\label{sec:appendages}

In this section, we introduce the main tool for the proof of the edge-version of Menger's Theorem, \cref{theo:edge_menger}. Given an arbitrary path $P$ of a bidirected graph $B$ starting in a vertex $x$, we define a set $A(P, x)$ of edges that includes much of the structure of $B$ in the vicinity of $P$:

\begin{definition} \label{def:appendage}
	An edge set $A \subseteq E(B)$ is {\em $(P,x)$-admissible} if for any $\ve \in \vAv$ there is an $\ve$--$x$ trail in $A \cup E(P)$.
	The {\em appendage} $A(P,x)$ of $P$ and $x$ is the union of all ${(P,x)}$-admissible sets.
\end{definition}

\begin{remark}
	Any union of $(P,x)$-admissible sets is again $(P,x)$-admissible.
	Thus, $A(P,x)$ is the maximal $(P,x)$-admissible set.
\end{remark}

In the vertex-version of Menger's Theorem, \cref{theo:vertex_menger}, we restrict to bidirected graphs~$B$ and sets~$\X$ of signed vertices such that~$B$ contains no nontrivial path that starts and ends in~$\X$.
However, for an edge-version of Menger's Theorem, forbidding nontrivial $x$--$x$~paths is not sufficient: 
joining the counterexample of \cref{thm:counterexample_to_edge_menger} with a bidirected graph akin to the one depicted in~\cref{fig:counterexample_extension} by identifying the graphs at the respective vertices called $x$ gives rise to a bidirected graph with no $x'$--$x'$ path.
But there are neither two disjoint $x'$--$y$ paths nor is there a small set $S$ of edges such that $B - S$ contains no~$x'$--$y$~path.

\begin{figure}[ht]
\begin{tikzpicture}
        \node[dot] (x') at (-2,0) [label=left:$x'$] {};
        \node[dot] (x) at (2,0) [label=right:$x$] {};
        \node[dot] (a1) at (0,1.2) {};
        \node[dot] (a2) at (0,.4) {};
        \node[dot] (a3) at (0,-.4) {};
        \node[dot] (a4) at (0,-1.2) {};
        \draw[-e+7,thick] (a1) ..controls+(1.6,0) .. (x) ;	
        \draw[-e+7,thick] (a2) ..controls+(1.1,0) .. (x) ;		
        \draw[-e+7,thick] (a3) ..controls+(1.1,0) .. (x) ;			
        \draw[-e+7,thick] (a4) ..controls+(1.6,0) .. (x) ;		
        \draw[-e+7,thick] (x') ..controls+(.7,1.2) .. (a1) ;	
        \draw[-e+7,thick] (x') ..controls+(.7,.4) .. (a2) ;
        \draw[-e+7,thick] (x') ..controls+(.7,-.4) .. (a3) ;	
        \draw[-e+7,thick] (x') ..controls+(.7,-1.2) .. (a4) ;			
    \end{tikzpicture}
    \caption{A bidirected graph with no $x'$--$x'$ path and the property that for any edge set $W$ of size at most $3$ there exists an $x'$-$x$ path in $B - W$.}
    \label{fig:counterexample_extension}
\end{figure}
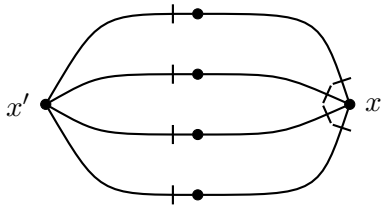

\noindent For the edge-version of Menger's Theorem, we thus forbid not only nontrivial $x$--$x$ paths, but even nontrivial $x$--$x$ trails:
\begin{definition}
	A vertex~$x$ of a bidirected graph~$B$ is \emph{edge-clean} if there exists no nontrivial $x$--$x$~trail in~$B$.
\end{definition}

From now on we assume $x$ to be edge-clean.
Our aim is to show three essential properties of the appendage $A(P, x)$ that makes it a key tool in our proof of~\cref{theo:edge_menger}, the edge-version of Menger's Theorem in bidirected graphs.
First, $A(P, x)$ is edge-disjoint from all paths starting in $x$ that are edge-disjoint from $P$.
Second, for any $v \in V(A(P, v)) \cup V(P)$ and any $\alpha \in \{ +, - \}$ if there is an $x$--$(v,\alpha)$ path in $B$ then there also is one in $A(P, x) \cup E(P)$ that coincides with $P$ in the first edge.
Third, any $x$--$y$ walk contains an $x$--$y$ path if it can be partitioned into a path in $A(P, x) \cup E(P)$ and a path in $E(B) \setminus (A(P,x) \cup E(P))$.
 
In the proof of~\cref{theo:edge_menger}, we replace the set $A(P,x) \cup E(P)$ by some auxiliary edges for a fixed path $P$. In this construction we do not want to remove paths starting in $x$ that are edge-disjoint to $P$. Therefore it is essential that all paths starting in $x$ which are disjoint to $E(P)$ are also disjoint to $A(P,x)$:

\begin{lemma}\label{lem:disjoint}
	Let $x$ be edge-clean and let $Q$ be a path in $B$ starting in $x$ that is edge-disjoint from $P$.
	Then $E(Q) \cap A(P,x) = \emptyset$.
\end{lemma}
\begin{proof}
	Suppose not for a contradiction, and let $\ve$ be the first edge of $Q$ in $A(P,x)$.
	Since $A(P,x)$ is $(P,x)$-admissible, there is an $\ve$--$x$-trail $R$ in $A(P,x) \cup E(P)$.
	But then $Q \ve R$
	is a nontrivial $x$--$x$~trail, contradicting the edge-cleanness of $x$.
\end{proof}

\begin{corollary} \label{cor:appendage_single_edge}
    Let $x$ be edge-clean.
    Then $A(P, x) \cup E(P)$ contains exactly one edge incident to $x$.
\end{corollary}
\begin{proof}
    Suppose not and let $e \in A(P, x) \setminus E(P)$ be incident to $x$.
    Then the path consisting of the single edge $e$ contradicts \cref{lem:disjoint}. 
\end{proof}
Now we turn our attention to the second essential property of appendages, which provides the existence of specific paths in $A(P, x) \cup E(P)$. We use this property in the proof of the edge-version of Menger's Theorem to ensure that certain paths are contained in $A(P, x) \cup E(P)$ by redirecting them if necessary. 
\begin{lemma}\label{lem:goodpath}
	Let $x$ be edge-clean and let $(v, \alpha)$ be a signed vertex with $v \in V(A(P,x)) \cup V(P)$.
	Suppose that there is an $x$--$(v,\alpha)$~path in $B$.
	Then there is such a path in $A(P,x) \cup E(P)$.
\end{lemma}
\begin{proof}
	Let $Q$ be an $x$--$(v,\alpha)$~path in $B$ chosen to minimise its set of edges outside $A(P,x) \cup E(P)$.
	We show that $Q$ is contained in $A(P,x) \cup E(P)$.
	Suppose not for a contradiction, and let $\ve$ be the first edge of $Q$ outside $A(P,x) \cup E(P)$. 
	\begin{description}
		\item[Case 1: Some vertex of $Q$ after $\ve$ lies on $P$]
		    In this case, let $w$ be the first such vertex along $P$.
		    Note that $w \neq x$ since $Q$ is a path.
		    Let $p$ be the edge preceding $w$ on $P$ and let~$q$ be the edge preceding $w$ on $Q$. 
    		\begin{description}
    			\item[Case 1.1: $\sigma(p,w) = \sigma(q,w)$]
    			    The path $P w Q$ is an $x$--$(v,\alpha)$~path using fewer edges outside $A(P,x) \cup E(P)$ than $Q$, since $P w$ is disjoint from $w Q$ per choice of $w$ and~$e \notin E(P w Q)$.
    			    This contradicts the minimality of $Q$.
    			    
    			\item[Case 1.2: $\sigma(p,w) = -\sigma(q,w)$]
        			We will show that the set $A(P,x) \cup E(Q w)$ is $(P,x)$-admissible.
        			That is, for any oriented edge $\vf$ with $f \in A(P,x) \cup E(Q w)$ there is an $\vf$--$x$~trail in $A(P,x) \cup E(Q w) \cup E(P)$.
        			If $f \in A(P,x)$ this is clear, so suppose not.
        			Thus, $f$ is an edge of $\ve Q w$.
        			If $\vf \in \vE(Q)$ then $\vf Q w \inv{P}$ is a suitable trail, by choice of~$w$.
                    If $\fv \in \vE(Q)$ then $\vf\,\inv{Q}$ is a suitable trail.
        			Thus, $A(P, x) \cup E(Q w)$ is $(P, x)$-admissible, and so it is a subset of $A(P, x)$.
        			This contradicts the assumption that $e$ is not contained in $A(P, x)$.
    		\end{description}
    		
		\item[Case 2: No vertex of $Q$ after $\ve$ lies on $P$]
    		In this case, let $w$ be the first vertex along $Q$ after $\ve$ in $V(A(P,x)) \cup V(P)$. There is such a vertex since $v$ is a candidate.
    		Let $q$ be the edge preceding $w$ on $Q$.
    		Our first aim is to show that there is $\va \in \vecev{A(P,x)}$ with startvertex $w$ and such that $\sigma(a,w) = -\sigma(q,w)$.
    		Indeed, as $w \in V(A(P,v))$ there is $\vb \in \vecev{A(P,x)}$ so that its endvertex is $w$.
    		If $\sigma(b,w) = -\sigma(q,w)$ then we can just set~$\va := \bv$.
    		Otherwise, we can take $\va$ as the second edge of any $\vb$--$x$ trail in~$A(P,x) \cup E(P)$.
    		Let $R$ be any $\va$--$x$ trail in~$A(P,x) \cup E(P)$.
    		
    		We will now show that $A(P,x) \cup E(Q w)$ is $(P,x)$-admissible, that is, for any oriented edge $\vf$ with $f \in A(P,x) \cup E(Q w)$, there is an $\vf$--$x$~trail in $A(P,x) \cup E(Q w) \cup E(P)$.
    		If~$f \in A(P,x)$ this is clear, so suppose not.
    		Thus, $f$ is an edge of $\ve Q w$. If~$\vf \in \vE(Q)$, then we find a suitable trail by following $Q$ from $\vf$ to $w$ and switching onto $\va R$. If~$\fv \in \vE(Q)$, then $\vf\,\inv{Q}$ is a suitable trail.
    		Thus, $A(P,x) \cup E(Q w)$ is $(P,x)$-admissible, and so it is a subset of $A(P,x)$.
    		This contradicts our assumption that $e$ is not contained in $A(P, x)$.
	\end{description}
\end{proof}

In the proof of~\cref{theo:edge_menger}, we will find some of the desired $x$--$y$ paths in certain \hbox{$x$--$y$~walks}.
As mentioned in the introduction, an~$x$--$y$~walk in a bidirected graph does not necessarily contain an~$x$--$y$~path.
We here provide a sufficient condition for the existence of an $x$--$y$ path in an~$x$--$y$ walk:
\begin{lemma}\label{lem:concat_paths}
    Let $y$ and $(z, \alpha)$ be (signed) vertices of $B$.
    Let $Q$ be an $x$--$(z, \alpha)$~path and let $R$ be a $(z, -\alpha)$--$y$~path.
    If all edges of $Q$ except possibly the last edge are contained in $E(P) \cup A(P,x)$ and $R$ avoids $E(P) \cup A(P,x)$, then there is an $x$--$y$~path $S$ in $E(Q)\cup E(R)$.
    Furthermore, if $R$ avoids $x$, then the first edge of $S$ and the first edge of $Q$ coincide.
\end{lemma}
\begin{proof}
    Let $a$ be the first vertex of $Q$ in $V(R)$.
    If $a = x$, then $a R$ is the desired path.
    If $a = z$, then $Q$ and $R$ intersect only in $a = z$ and thus $Q z R$ is the desired path. We can suppose that $a$ is an internal vertex of $Q$.
    Let $e$ be the edge of $Q$ preceding $a$ and $f$ be the edge of $R$ succeeding $a$.
    If $\sigma(a, e) = - \sigma(a, f)$, then $Q a R$ forms the desired path.

    Suppose for a contradiction that $\sigma(a, e) = \sigma(a, f)$ holds.
    Then $a \inv{R} z \inv{Q} x$ and $a Q z R a \inv{Q} x$ are trails.
    We show that $E(R a) \subseteq A(P, x)$, which contradicts the assumption that $R$ avoids $A(P, x)$.
    More precisely, we prove that $A(P, x) \cup E(a Q) \cup E(R a)$ is $(P,x)$-admissible: Let $\ve \in \vecev{E(a Q)} \cup \vecev{E(R a)}$.
    Then $\ve \in \vE(a \inv{R} z \inv{Q} x)$ or $\ve \in \vE(a Q z R a \inv{Q} x)$ holds.
    Note that both trails are contained in $A(P, x) \cup E(a Q) \cup E(R a)$, as $E(Q a) \subseteq A(P, x)$ holds since $a \neq z$.
    This completes the proof.
\end{proof}

\section{Menger's Theorem for edge-disjoint paths}\label{sec:edges-disjoint}

In this section we prove an edge-version of Menger's Theorem in bidirected graphs, from which we will then deduce our main result, the vertex-version given by~\cref{theo:vertex_menger}, in the subsequent~\cref{sec:vertex-disjoint}:
\begin{theorem}\label{theo:edge_menger}
    Let~$x$ and~$y$ be distinct vertices of a bidirected graph~$B$, and suppose that~$x$ is edge-clean.
    Let~$P_1, \dots, P_k$ be edge-disjoint~$x$--$y$~paths in~$B$ where~$\ve_i$ is the first edge of~$P_i$ for~$i \in [k]$.
    Then precisely one of the following is true:
	\begin{enumerate}[label=(\arabic*)]
		\item\label{item:edge_menger1} There is a set $S$ of $k$ edges of $B$ such that $B - S$ contains no $x$--$y$~path. \label{item:edge-sep}
		\item\label{item:edge_menger2} There are~$k+1$ edge-disjoint $x$--$y$~paths $P_1', \dots, P_{k+1}'$ such that the first edge of $P_i'$ is $\ve_i$ for $i \in [k]$. \label{item:edge-paths}
	\end{enumerate}
\end{theorem}
The proof of \cref{theo:edge_menger} is inspired by the B\"ohme, G\"oring and Harant's proof \cite{Boehme01} of Menger's Theorem for directed graphs.
Its technical finesse makes it possible to handle the complex structure of bidirected graphs.
\begin{proof}
	The proof is by strong induction on the sum of the lengths of the paths $P_i$.
	We may assume without loss of generality that $\sigma(e,x) = -$ for any edge $e$ incident with $x$, since changing these signs doesn't affect what counts as an $x$--$y$~path nor the edge-cleanness of~$x$.
	
	If the set $\{e_1, e_2, \ldots , e_k\}$ is as in \labelcref{item:edge_menger1} then we are done, so suppose not.
	Thus, there is an~$x$--$y$~path $P_{k+1}$ containing no $e_i$.
	If this path is edge-disjoint from the $P_i$ with $i \leq k$ then we are done, so suppose not.
	Let $\ve$ be the first edge of $P_{k+1}$ which lies on some other $P_i$.
    Without loss of generality we may assume that $e \in E(P_k)$.
	If $\ev$ were in $\vE(P_k)$ then $P_{k+1}\ve \inv{P}_k$ would be an $x$--$x$~trail, contradicting edge-cleanness of $x$.
	So we must instead have $\ve \in \vE(P_k)$.
	Let $v$ be the startvertex of $\ve$ and $w$ its endvertex.
	We define $R_1$ as $P_k v$ and $R_2$ as $P_{k+1}v$.
    Note that $R_1$ and $R_2$ have length at least one.
	
	All we will use about $R_2$ in the following argument is that it is edge-disjoint from all $P_i$ and that it can be extended as a path by adding the edge $e$.
	Note that if we had begun with the paths $P_1, \dots, P_{k-1}$ and $R_2 v P_k$ then $R_1$ would have the same properties with respect to this choice of paths.
	We will exploit this symmetry repeatedly in the following argument.
	
	Let $A_i := A(R_i, x) \cup E(R_i)$.
	Note that by~\cref{lem:disjoint} $A_1$ is disjoint from $E(R_2 v P_k)$ and $A_2$ is disjoint from $E(R_1 v P_k)$.
	Similarly both $A_1$ and $A_2$ are disjoint from $E(P_i)$ for $i < k$.
    To apply the induction hypothesis we construct a suitable graph $\hat{B}$ using the sets $A_1$ and $A_2$.
	
	We obtain $\hat{B}$ from $B$ by modifying it in the following ways (see \cref{fig:construction_aux_graph}):

	\begin{itemize}
		\item removing the edge $e$,
		\item adding a new edge $\hat{e}$ with endvertices $x, w$ and $\sigma(\hat{e}, x) = -$ and $\sigma(\hat{e}, w) = \sigma(e,w)$,
		\item removing all edges in $A_1 \cup A_2$,
		\item adding a new vertex $a$,
		\item adding a new edge $e_a$ from $x$ to $a$ with $\sigma(e_a, x) = -$ and $\sigma(e_a, a) = +$,
		\item adding, for each $i \in \{1,2\}$ and each signed vertex $(z, \alpha)$ for which there is a nontrivial $x$--$(z, \alpha)$~path $P$ with $E(P) \subseteq A_i$ (not just a trail) a new edge $f$ with endvertices $a, z$ and $\sigma(a, f) = -$ and with $\sigma(z, f)= \alpha$. We refer to $f$ as $e(P)$ and to $z$ as $z_P$.
	\end{itemize}
	
	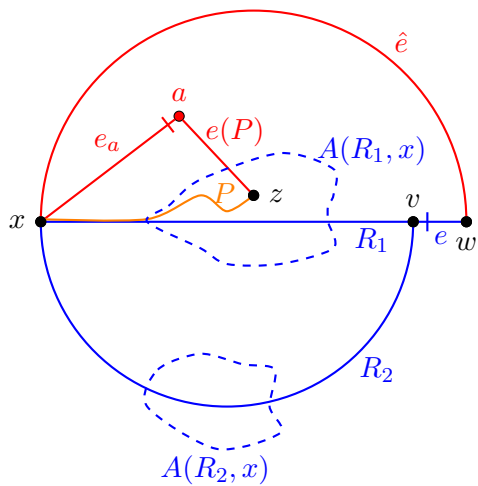
\begin{figure}
 	\begin{tikzpicture}[scale=.7]
		\node[dot] (xx) at (-16,6) [label={[]left:$x$}] {};
		\node[red,dot] (a) at (-13.4,8) [label={[red]above:$a$}] {};
		\node[dot] (cc) at (-8,6) [label={[]below:$w$}] {};
		\node[dot] (bb) at (-9,6) [label={[]above:$v$}] {};
  	    \node[dot]  (z) at (-12,6.5) [label={right:$z$}] {};
		\draw [thick, draw=blue] (xx) to (bb);
            \draw [-e+,thick, draw=blue] (cc) to (bb); 
		\draw [-e+, thick,color=red] (xx) to (a); 
		\draw [thick, draw=red] (a) to (z); 
  	\draw[thick, draw=red] (cc) arc (0:180:4);
		\draw[thick, draw=blue] (bb) arc (0:-180:3.5);
		\draw [thick, draw=orange]  plot [smooth] coordinates {(-16,6.05) (-14,6.05) (-13.,6.5) (-12.5,6.2) (z)};
		\draw [thick, dashed, draw=blue]  plot [smooth cycle] coordinates {(-14.0,6) (-13.6,6.3) (-13,6.8) (-11.5,7.3) (-10.5,7) (-10.5,6.5) (-10.6,6.3) (-10.6,6)  (-10.5,5.5) (-11.5,5.2) (-12.5,5.2) (-13,5.5) };
		\draw [thick, dashed, draw=blue]  plot [smooth cycle] coordinates {(-14.,2.5)(-13.8,3.2) (-13,3.5)  (-12.2,3.3)  (-11.6,3.2) (-11.7,2.5) (-11.5,2) (-12.7,1.7) };
			\node[red,left] at (-8.9,9.4) {$\hat e$};
		\node[red,left] at (-14.3,7.5) {$e_a$};
		\node[red,left] at (-11.6,7.7) {$e(P)$};
		\node[orange,left] at (-12.15,6.55) {$P$};
		\node[blue,right] at (-11,7.35){$A(R_1,x)$};
		\node[blue,right] at (-10.3,5.65) {$R_1$};
		\node[blue,right] at (-8.8,5.7){$e$};
		\node[blue,right] at (-10.2,3.2){$R_2$};
		\node[blue,left] at (-11.5,1.3){$A(R_2,x)$};
  \foreach \x in {xx,cc,z,bb} \node[dot] at (\x) {};
	\end{tikzpicture}
	    \caption{Construction of the auxiliary graph $\hat{B}$ in the proof of \cref{theo:edge_menger}.}
	    \label{fig:construction_aux_graph}
	\end{figure}
	
    We set $\vehat$ and $\ve_a$ to be the orientation of $\hat{e}$ and $e_a$ that point away from $x$.
    Furthermore, we set $\ve(P)$ to be the orientation of $e(P)$ that points away from $a$ and $\ve$ the orientation of $e$ that points away from $v$.
    For $i < k$ let $\hat{P}_i := P_i$, and let $\hat{P}_k := \vehat w P_k$.
	Note that $\hat{P}_k$ is shorter than $P_k$.
	
	We now want to apply the induction hypothesis to $\hat{B}$ and the paths $\hat{P}_i$ for $i \leq k$.
	Since the sum of the lengths of the paths has decreased, we just need to check that $x$ is still edge-clean in $\hat{B}$.
	So suppose for a contradiction that there is an $x$--$x$~trail $\hat{Q}$ in $\hat{B}$. 
	
	Since all edges incident to $x$ are incident with the same sign, no internal edge of $\hat{Q}$ is incident with $x$.
	Similarly, since $e_a$ is the only edge incident to $a$ with sign $+$, if $E(\hat{Q})$ contains any~$e(P)$, then that $e(P)$ must be adjacent to $e_a$, which in turn must be the first or last edge of~$\hat{Q}$.
	
	Our aim now is to construct a $x$--$x$~trail $Q$ in $B$, thus contradicting the edge-cleanness of $x$ in $B$.
	Let $\vq_1$ be the first edge and $\qv_2$ the last edge of $\hat{Q}$.
	Replacing $\hat{Q}$ with its reversal and relabelling these edges if necessary, we have the following cases:
	
	\begin{description}
		\item[Case 1: Both $q_1$ and $q_2$ are edges of $B$]
		    We can set $Q := \hat{Q}$.
		
		\item[Case 2: $q_1 = \hat{e}$ and $q_2$ is an edge of $B$]
	    	We can set $Q := P_k w \hat{Q}$. 
		
		\item[Case 3: $q_1 = e_a$ and $q_2$ is an edge of $B$]
    		The second edge of $\hat{Q}$ must be of the form $\ve(P)$ for some nontrivial path $P$ in some $A_i$.
    		So we can set $Q := P z_P \hat{Q}$.
		
		\item[Case 4: $q_1 = e_a$ and $q_2 = \hat{e}$]
    		The second edge of $\hat{Q}$ must be of the form $\ve(P)$ for some nontrivial path $P$. Without loss of generality $E(P)$ is contained in $A_1$.
      So we can set $Q := P z_P \hat{Q} w \cev{e} v \inv{R}_2$.
	\end{description}
	
	In any case we reach the desired contradiction, so we can conclude that $x$ is still edge-clean in $\hat{B}$. 
	Thus, we can apply the induction hypothesis.
	This gives us two cases:
	
	\begin{description}
		\item[Case 1:] \textbf{There is a set $\hat{S}$ of $k$ edges of $\hat{B}$ such that $\hat{B} - \hat{S}$ does not contain an $x$--$y$~path:}
		    Since the paths $\hat{P}_i$ are edge-disjoint, $\hat{S}$ must consist of one edge from each of these paths and so cannot contain $e_a$ or any $e(P)$.
		    Let $S := \hat{S}$ if $\hat{e} \not \in \hat{S}$ and let $S := (S \setminus \{\hat{e}\}) \cup \{e\}$ otherwise.
		    Clearly $S$ has size $k$.
		    We show that $B-S$ contains no $x$--$y$~path. 
		
		Suppose for a contradiction that there is such a path $Q$.
		To obtain our contradiction we construct an $x$--$y$~path $\hat{Q}$ in $\hat{B} - \hat{S}$. 
		\begin{description}
			\item[Case 1.1: $E(Q)$ is disjoint from $A_1 \cup A_2 \cup \{e\}$]
			    We can set $\hat{Q} := Q$.
			
			\item[Case 1.2: $\ve \in \vE(Q)$ and $E(w Q)$ is disjoint from $A_1 \cup A_2$]
			    We have $e \not \in S$, which implies~$\hat{e} \not \in \hat{S}$.
			    Thus, we can set $\hat{Q} := x \vehat w Q$.
			
			\item[Case 1.3: Otherwise]
    			Since $v \in V(A_1)$ the fact that we are not in Case 1.1 implies that $Q$ meets $V(A_1 \cup A_2) \setminus \{x\}$.
    			Let $z$ be the last vertex of $Q$ in this set.
       Without loss of generality we have $z \in V(A_1) \setminus \{x\}$.
    			The fact that we are not in Case 1.2 implies that even if $e$ appears on $Q$ some vertex in $V(A_1 \cup A_2) \setminus \{x\}$ (possibly $v$) must come after it.
    			So the set $E(z Q)$ does not contain $e$.
       Let $\alpha$ be the sign with which $Q$ arrives at $z$.
       Applying~\cref{lem:goodpath} to the $x$--$(z, \alpha)$ path $Q z$ we see that there must be an $x$--$(z, \alpha)$ path $P$ in $A_1$.
    			Then we can set $\hat{Q} := x \ve_a a \ve(P) z Q$.
		\end{description}
		In each case, we find the desired contradiction, completing the proof of this case.
		
		\item[Case 2] \textbf{There are edge-disjoint $x$--$y$~paths $\hat{P}_1',\hat{P}_2', \dots, \hat{P}_{k+1}'$ such that the first edge of $\hat{P}_i'$ is $\ve_i$ for $i < k$ and the first edge of $\hat{P}_k'$ is $\vehat$:}
    		In this case, we will construct edge-disjoint $x$--$y$~paths $P_1',P_2',\dots, P_{k+1}'$ in $B$ such that the first edge of $P_i'$ is $\ve_i$ for~$i \leq k$.

            For $i < k$ the path $\hat{P}_i'$ avoids the edges $\hat{e}$ and $e_a$ since $\ve_i$ is the first edge of $\hat{P}_i'$, and therefore it avoids also any edge $e(P)$.
            Thus, $\hat{P}_i'$ is a path of $B$, and we set $P_i:= \hat{P}_i'$ for any $i < k$.
            It remains to define $P_k'$ and $P_{k + 1}'$.
            \begin{description}
                \item[Case 2.1: The first edge of $\hat{P}_{k + 1}'$ is not $\ve_a$] Then $\hat{P}_{k + 1}'$ is also a path of $B$ as it avoids $\hat{e}$ since $\hat{e} \in P_k$.
                We set $P_{k + 1}':=\hat{P}_{k + 1}'$ and let $P_k'$ be the path obtained by applying \cref{lem:concat_paths} to $R_1 v \ve w$ and $ w\hat{P}_k'$ with respect to $A_1$.
                \item[Case 2.2] \textbf{The first edge of $\hat{P}_{k + 1}'$ is $\ve_a$ and its second edge is $\ve(P)$ for some path $P$ in $A_2$:} We let $P_{k+1}'$ be the path obtained by applying \cref{lem:concat_paths} to $P z_P, z_P \hat{P}_{k+1}'$ with respect to $A_2$ and define $P_k'$ as in the prior case.
                \item[Case 2.3] \textbf{The first edge of $\hat{P}_{k + 1}'$ is $\ve_a$ and its second edge is $\ve(P)$ for some path $P$ in $A_1$:} Let $P_k'$ be the path that we obtain by applying \cref{lem:concat_paths} to~$P z_P,  z_P \hat{P}_{k+1}'$ with respect to $A_1$ and let $P_{k+1}'$ be the path obtained by applying \cref{lem:concat_paths} to~$R_2 v \ve w$ and $w \hat{P}_k'$ with respect to $A_2$.
            \end{description}
            Note that in any of these cases the paths $P_1', \dots, P_{k +1}'$ are edge-disjoint.
            To verify that $\ve_k$ is indeed the first edge of $P_k'$, it suffices to show that in the construction of $P_k'$ we applied \cref{lem:concat_paths} to a path whose first edge is $\ve_k$ and a path that avoids $x$.
            The first path is $P z_P$, if $P$ is in $A_1$, and otherwise it is $R_1 v \ve w$. Thus, the first edge of the first path is contained in $A_1$, and by \cref{cor:appendage_single_edge} it is $\ve_k$.
            The second path avoids the vertex $x$ by choice. Thus, the paths $P_1', \dots, P_{k +1}'$ are as desired.
	\end{description}
 
\end{proof}

\cref{theo:edge_menger} implies Menger's Theorem for edge-disjoint paths in both undirected and directed graphs.
To show this, we may regard any graph as a directed graph by replacing each edge~$e$ with the two directed edges~$\ve$ and~$\ev$ (see~\cite{bang2008} for an in-depth explanation).
Similarly, we may regard any directed graph as a bidirected graph by viewing a directed edge~$\ve$ from~$x$ to~$y$ as an edge~$e$ with endvertices~$x$ and~$y$ and signs~$-$ at~$x$ and~$+$ at~$y$.

We then obtain the desired version of Menger's Theorem for vertices $x$ and $y$ in a directed graph~$D$ by considering the directed graph $D'$ obtained from $D$ by deleting all edges that point towards~$x$.
Note that~$x$ is edge-clean in $D'$ since no trail ends in $x$.
If there is no $x$--$y$ path in $D' - S$ for some $S \subseteq V(D')$, then there is no $x$--$y$ path in $D - S$.
Furthermore, any~$x$--$y$~path in $D'$ is also an $x$--$y$ path in $D$.

\section{Menger's Theorem for vertex-disjoint paths}\label{sec:vertex-disjoint}

In this section we deduce the vertex-disjoint version of Menger's Theorem for bidirected graphs, \cref{theo:vertex_menger}, from the above shown edge-version, \cref{theo:edge_menger}.
\cref{theo:vertex_menger} will be formulated in terms of vertex-disjoint~$\X$--$\Y$~paths between two sets~$\X$ and~$\Y$ of signed vertices of a bidirected graph~$B$.
Analogously to the edge-version where the startvertex~$x$ of the paths had to be edge-clean, we have to require the set~$\X$ of signed vertices to be \emph{clean}, which is defined as follows.

\begin{definition}
    A set~$\X$ of signed vertices of a bidirected graph~$B$ is \emph{clean} if $B$ contains no nontrivial path starting and ending in~$\X$.
\end{definition}

\noindent We remark that a clean set~$\X$ of signed vertices may contain both~$(x, +)$ and~$(x, -)$ for a vertex~$x$.

Let us now recall the vertex-disjoint version of Menger's Theorem for bidirected graphs from the introduction:

\VertexMenger*

\noindent Note that in~\cref{theo:vertex_menger} we only require $\X$ to be clean, but not $\Y$ -- just as we only require~$x$ (and not $y$) to be edge-clean in \cref{theo:edge_menger}; indeed, we do not need any assumptions on~$\Y$.
We further note that~\cref{theo:vertex_menger} cannot be strengthened by fixing the signed startvertices of the paths~$P_i$ rather than just their startvertices, see~\cref{fig:counterexample_cannot_fix_signed_vertices_in_vertex_menger} for a counterexample.

\begin{figure}[ht]
    \centering
    \begin{tikzpicture}
        \node[dot] (x1) at (0,1) [label={left:$x_1$}] {};
        \node[dot,below=1cm] (x2) at (x1) [label=left:$x_2$] {};
        \node[dot,right=2cm] (y1) at (x1) [label=right:$y_1$] {};
        \node[dot,below=1cm] (y2) at (y1) [label=right:$y_2$] {};
        \node[shape = coordinate,xshift = -.02cm,yshift=.01cm] (start) at (x1) {};
        \node[shape = coordinate,xshift = .02cm,yshift=-.01cm] (finish) at (y2) {};
        \node[red,xshift=-1.3cm,yshift=.26cm] at (x1) {$(x_1,+)$};
        \node[red,xshift=-1.3cm,yshift=-.2cm] at (x1) {$(x_1,-)$};
        \node[red,xshift=-1.3cm,yshift=.03cm] at (x2) {$(x_2,+)$};
        \node[red,xshift=1.3cm,yshift=.03cm] at (y1) {$(y_1,+)$};
        \node[red,xshift=1.3cm,yshift=.03cm] at (y2) {$(y_2,-)$};
        \node[red,xshift=-1.3cm,yshift=.8cm] at (x1) {$\X$};
        \node[red,xshift=1.3cm,yshift=.8cm] at (y1) {$\Y$};
        \draw[line width=9pt, blue!60!white, line cap=round] (start) to node[pos=.8,above] {$P_1$} (finish);
        \draw[line width=8pt, blue!30!white, line cap=round] (start) to (finish);
        \foreach \startvertex/\endvertex in {x1/y1,y2/x1,y2/x2} \draw[-e+7,thick] (\startvertex) to (\endvertex) ;
        \node[dot] at (x1) {};
        \node[dot] at (y2) {};
    \end{tikzpicture}
    \caption{
    	This bidirected graph contains two disjoint~\textcolor{red}{$\X$}--\textcolor{red}{$\Y$}~paths, namely~$x_2 (x_2y_2,x_2,y_2) y_2$ and $x_1 (x_1y_1,x_1,y_1)y_1$.
    	However, while the \textcolor{red}{$\X$}--\textcolor{red}{$\Y$}~path \textcolor{blue!60!white}{$P_1 := x_1 (x_1y_2,x_1,y_2) y_2$} starts in~\textcolor{red}{$(x_1,+) \in \X$}, the graph contains no two disjoint \textcolor{red}{$\X$}--\textcolor{red}{$\Y$}~paths such that one of them starts in~\textcolor{red}{$(x_1,+)$}.
    }
    \label{fig:counterexample_cannot_fix_signed_vertices_in_vertex_menger}
\end{figure}
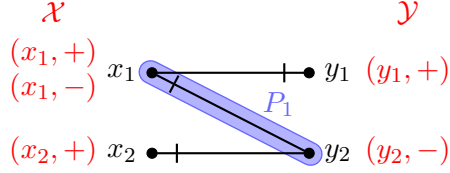

Towards a proof of~\cref{theo:vertex_menger}, let us first show that we may assume the~$\X$--$\Y$~paths in~$B$ to be the only paths starting in~$\X$ and ending in~$\Y$, respectively.
More precisely, by removing some edges from~$B$ and changing some signs at vertices of~$V(\X) \cup V(\Y)$, all nontrivial paths starting in~$\X$ and ending in~$\Y$ are indeed~$\X$--$\Y$ paths, i.e. they are internally disjoint from~$V(\X) \cup V(\Y)$ and contain no trivial $\X$--$\Y$~paths:

\begin{proposition} \label{prop:SimplifyXYpaths}
    Let $B$ be a bidirected graph and $\X, \Y$ sets of signed vertices in~$B$.
    Then there exists a bidirected graph $B'$ with $V(B') = V(B)$, $E(B') \subseteq E(B)$ and a set of signed vertices $\X' \subseteq \X$ in $B'$ such that
    \begin{enumerate}[label=(\arabic*)]
        \item\label{item:graph_appearance_assumption_paths_preserved_1} any $\X'$--$\Y$~path in $B'$ is an $\X$--$\Y$~path in $B$ and vice versa,
        \item\label{item:graph_appearance_assumption_X-Y_path_is_immediate_1} any path in $B'$ starting in $\X'$ and ending in $\Y$ is an $\X'$--$\Y$~path, and
        \item\label{item:graph_appearance_assumption_X-X_path_is_immediate_2} there is no trivial $\X'$--$\X'$ path in $B'$.
    \end{enumerate}
    Moreover, if $\X$ is clean in $B$, then $\X'$ is clean in $B'$.
\end{proposition}

\noindent Note that, by \labelcref{item:graph_appearance_assumption_paths_preserved_1}, \cref{theo:vertex_menger}, or more generally any version of Menger's theorem concerning vertex-disjoint~$\X$--$\Y$~paths, holds for~$B$, $\X$ and~$\Y$ if and only if it holds for~$B'$, $\X'$ and~$\Y$.

\begin{proof}
    We construct $B'$ and $\X'$ from $B$ and $\X$ by performing the following modifications for every~$v \in V(\X) \cup V(\Y)$ simultaneously:
    \begin{enumerate}[label=(\alph*)]
        \item\label{item:modification_a} If there is $\alpha \in \{+, -\}$ with $(v, \alpha) \notin \X \cup \Y$, we delete all edges incident to $v$ that have sign $\alpha$ at $v$.
        \item\label{item:modification_b} If $(v, +)$ and $(v, -)$ are either both contained in $\X \setminus \Y$ or both contained in $\Y \setminus \X$, we change the signs of all edges incident to $v$ at $v$ to $+$.
        Furthermore, if $(v, -) \in \X$, then we remove $(v, -)$ from $\X$.
        \item\label{item:modification_c} If there is $\alpha \in \{+, -\}$ such that $(v, \alpha) \in \X$ and $(v, -\alpha) \in \Y$, we isolate $v$, i.e.\ we delete all edges incident with $v$.
        If $(v, -\alpha) \in \X$, then we also delete $(v, -\alpha)$ from $\X$ .
    \end{enumerate}
    With this construction, we in particular obtain the following properties of~$B'$:
    
    \begin{enumerate}[label=(\roman*)]
        \item\label{item:construction_of_B_prime_1} $V(\X) \cup V(\Y) = V(\X') \cup V(\Y )$, and all edges incident to $v \in V(\X') \cup V(\Y )$ have the same sign at $v$ in $B'$.
        \item\label{item:construction_of_B_prime_2} If an edge of $B'$ is incident to a vertex $w \notin V(\X') \cup V(\Y)$, then it has the same sign at $w$ in~$B'$ and in~$B$.
        \item\label{item:construction_of_B_prime_3} If a vertex $v$ forms a trivial $\X$--$\Y$~path in $B$, then it is an isolated vertex in $B'$ and forms a trivial $\X'$--$\Y$~path in $B'$.
    \end{enumerate}
        
    We begin by proving \labelcref{item:graph_appearance_assumption_paths_preserved_1}. For trivial paths, the statement is true by \labelcref{item:construction_of_B_prime_3} and since~$\X' \subseteq \X$.
    So let $P$ be a nontrivial $\X$--$\Y$~path in $B$.
    Then neither its start- nor its endvertex forms a trivial $\X'$--$\Y$~path in $B'$ since they each do not form a trivial $\X$--$\Y$~path in $B$.
    In particular, we did not apply modification~\labelcref{item:modification_c} to the start- or endvertex of~$P$.
    As all internal vertices of $P$ are disjoint to $V(\X) \cup V(\Y)$, we did not apply any modification to these vertices.
    This implies that all edges of $P$ are edges in $B'$.
    By \labelcref{item:construction_of_B_prime_2} $P$ is thus also a path in $B'$.
    Clearly, it is internally disjoint to~$V(\X') \cup V(\Y)$.
    Finally, since we applied either modification~\labelcref{item:modification_a} or \labelcref{item:modification_b} to its start- and endvertex, the path~$P$ in $B'$ starts in $\X'$ and ends in $\Y$.
    
    Let $P'$ be a nontrivial $\X'$--$\Y$~path in $B'$.
    Then the path $P'$ is, by \labelcref{item:construction_of_B_prime_1} internally disjoint to $V(\X) \cup V(\Y) = V(\X') \cup V(\Y)$ and thus also a path in $B$ by \labelcref{item:construction_of_B_prime_2}.
    As~ \labelcref{item:graph_appearance_assumption_paths_preserved_1} holds for trivial paths, neither its start- nor its endvertex forms a trivial $\X$--$\Y$~path in $B$. In particular, we did not apply modification~\labelcref{item:modification_c} to its start- or endvertex.
    Thus, we applied either modification~\labelcref{item:modification_a} or \labelcref{item:modification_b} to its start- and endvertex.
    This implies that $P'$ in $B$ starts in~$\X$ and ends in~$\Y$, which completes the proof of \labelcref{item:graph_appearance_assumption_paths_preserved_1}.
    
    For \cref{item:graph_appearance_assumption_X-Y_path_is_immediate_1}, consider a nontrivial path~$P'$ in $B'$ that starts in $\X'$ and ends in $\Y$.
    To show that~$P'$ is an~$\X'$--$\Y$ path, we need to verify two properties:
    First, no internal vertex of~$P'$ is contained in~$V(\X')\cup V(\Y)$, which is true by \cref{item:construction_of_B_prime_1}.
    Second, $P'$ contains no trivial $\X'$--$\Y$ path.
    Indeed, if a vertex~$v \in P'$ is a trivial~$\X'$--$\Y$~path in~$B'$, then it is also a trivial~$\X$--$\Y$~path in~$B$ by~\cref{item:graph_appearance_assumption_paths_preserved_1} and hence isolated by~\cref{item:construction_of_B_prime_3}, which contradicts that~$P'$ is nontrivial.

    For \cref{item:graph_appearance_assumption_X-X_path_is_immediate_2}, note that the modifications \labelcref{item:modification_a,item:modification_b,item:modification_c} ensure that $(v,\alpha) \in \X'$ implies~$(v,-\alpha) \notin \X'$.

    For the `moreover'-part, assume that $\X$ is clean in~$B$ and suppose for a contradiction that~$x$ is not edge-clean in~$B'$.
    Then there exists a nontrivial path $P$ in $B'$ that starts and ends in~$\X'$.
    By \cref{item:construction_of_B_prime_1}, no internal vertex of $P$ is contained in~$V(\X) \cup V(\Y)$.
    Furthermore, by \cref{item:construction_of_B_prime_2}, $P$ still forms a path in~$B$.
    Thus, $P$ is a nontrivial path in $B$ starting and ending in~$\X$, which contradicts the cleanness of~$\X$ in~$B$.
\end{proof}

From now on, let~$B$ be a bidirected graph and let~$\X$ and~$\Y$ be sets of signed vertices of~$B$.
By~\cref{prop:SimplifyXYpaths} we may assume that~$B$ satisfies \labelcref{item:graph_appearance_assumption_X-Y_path_is_immediate_1,item:graph_appearance_assumption_X-X_path_is_immediate_2}.
To prove~\cref{theo:vertex_menger} for~$B$, $\X$ and~$\Y$, we now carefully construct an auxiliary graph~$\hat{B}$ with two specified vertices~$x$ and~$y$, such that~\cref{theo:edge_menger}, our edge-disjoint version of Menger's theorem in bidirected graphs, applied to~$\hat{B}$, $x$ and~$y$ allows us to deduce~\cref{theo:vertex_menger} for~$B$, $\X$ and~$\Y$.

The construction of~$\hat{B}$, however, turns out to be even more general:
It provides a framework to deduce a vertex-version of Menger's theorem in bidirected graphs from an edge-version.
In particular, our construction of~$\hat{B}$ is independent of the cleanness assumption on~$\X$ as in~\cref{theo:vertex_menger}, and the cleanness of~$\X$ in~$B$ will only be used to make~\cref{theo:edge_menger} applicable to~$x$ in~$\hat{B}$ in that it implies~$x$ to be edge-clean.

In order to make the above approach work, we construct~$\hat{B}$, $x$ and~$y$ such that the vertices~$x$ and~$y$ of~$\hat{B}$ represent the sets~$\X$ and~$\Y$ of signed vertices of~$B$, respectively.
They do so in such a way that any set of vertex-disjoint $\X$--$\Y$~paths in $B$ correspond to edge-disjoint $x$--$y$~paths in $\hat{B}$ and vice versa.

To show that a vertex-version of Menger's theorem for~$B$, $\X$ and~$\Y$ can be deduced from an edge-version of Menger's theorem for~$\hat{B}$, $x$ and~$y$, we then verify two properties of our construction.
We first show that an edge-separator for~$x$ and~$y$ in~$\hat{B}$ corresponds to a vertex-separator for $\X$ and $\Y$ in $B$ of at most the same size.
Secondly, we consider vertex-disjoint $\X$--$\Y$~paths $P_1, \dots, P_k$ in $B$.
By our construction, they correspond to edge-disjoint $x$--$y$~paths $\hat{P}_1, \dots, \hat{P}_k$ in $\hat{B}$.
We then show that $k + 1$ edge-disjoint $x$--$y$~paths in $\hat{B}$ such that all but one start in the same edges as $\hat{P}_1, \dots, \hat{P}_k$ correspond to $k + 1$ vertex-disjoint $\X$--$\Y$~paths such that all but one start in the same vertices as $P_1, \dots, P_k$.

Finally, we show that our construction indeed allows us to apply~\cref{theo:edge_menger} to~$\hat{B}$, $x$ and~$y$ under the assumptions of~\cref{theo:vertex_menger} on~$B$, $\X$ and~$\Y$.
More precisely, we prove that the cleanness of~$\X$ in~$B$ implies that~$x$ is edge-clean in~$\hat{B}$. \\

Let us now construct~$\hat{B} = (\hat{G}, \hat{\sigma})$ from~$B = (G, \sigma)$ (see~\cref{fig:Construction_of_auxiliary_graph} for an illustration of this construction).
For every vertex~$v \in B$, the vertex set of $\hat{B}$ contains two vertices $v^+$ and $v^-$ which are connected in~$\hat{B}$ by an edge with sign~$+$ at~$v^-$ and sign~$-$ at $v^+$.
An edge~$e \in E(B)$ which is incident to~$u$ and~$v$ in~$B$ then transfers to an edge $\hat{e} \in E(\hat{B})$ that is incident to~$u^{\sigma(u,e)}$ and $v^{\sigma(v,e)}$.
The half-edges of~$\hat{e}$ have the same signs as the respective half-edges of~$e$, that is $\hat{\sigma}(u^{\sigma(u,e)},\hat{e}):=\sigma(u,e)$ and~$\hat{\sigma}(v^{\sigma(v,e)},\hat{e}):=\sigma(v,e)$.

We finally add two new vertices~$x$ and~$y$ to~$V(\hat{B})$, which shall represent the sets~$\X$ and~$\Y$, respectively, as follows.
The vertex~$x$ is adjacent to those~$v^\alpha \in V(\hat{B})$ with~$(v, -\alpha) \in \X$, and the respective edge in~$\hat{B}$ has sign~$\alpha$ at $v^\alpha$ and sign~$-$ at~$x$.
Analogously, the vertex~$y$ is adjacent to every~$v^\beta \in V(\hat{B})$ with $(v, -\beta) \in \Y$, and the respective edge has sign~$\beta$ at~$v^\beta$ and sign~$-$ at~$y$.

Observe that with this construction of~$\hat{B} = (\hat{G}, \hat{\sigma})$, the graph~$\hat{G}$ has no parallel edges.
For notational simplicity, we will thus identify an (oriented) edge in~$\hat{B}$ with the (ordered) pair of its endvertices in what follows.

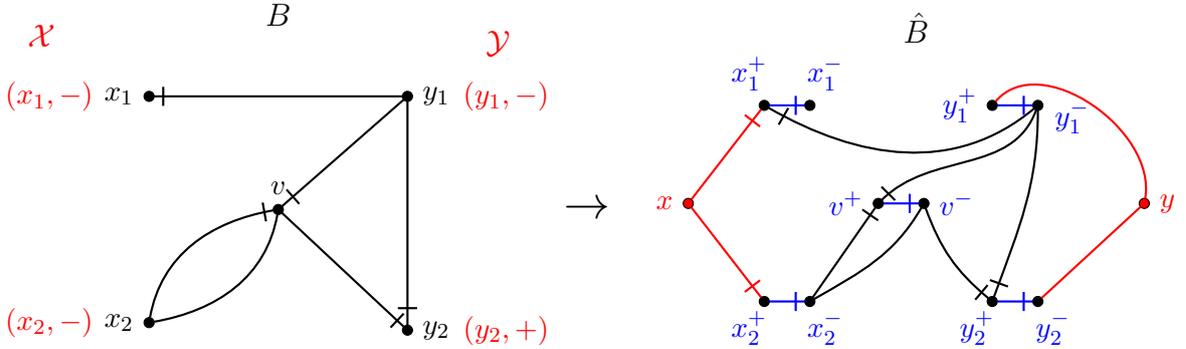
\begin{figure}[htb]
	\centering
	\begin{tabular}{ccc}
       	\begin{tikzpicture}
           	\node[dot] (v) at (0,0) [label=above:$v$] {};
           	\node[dot] (y1) at (1.7,1.5) [label=right:$y_1$] {};
           	\node[red,right] at (2.3,1.5) {$(y_1,-)$};
           	\node[dot] (y2) at (1.7,-1.6) [label=right:$y_2$] {};
           	\node[red,right] at (2.3,-1.6) {$(y_2,+)$};
           	\node[red,right] at (2.6,2.2) {\large$\Y$};
           	\node[dot] (x1) at (-1.7,1.5) [label=left:$x_1$] {};
           	\node[red,left] at (-2.3,1.5) {$(x_1,-)$};
           	\node[dot] (x2) at (-1.7,-1.5) [label=left:$x_2$] {};
           	\node[red,left] at (-2.3,-1.5) {$(x_2,-)$};
           	\node[red,left] at (-2.8,2.3) {\large$\X$};
           	\node[above] at (0,2.3) {\large$B$};
           	\draw[-e+4, thick] (y1) to (v);
           	\draw[-e+, thick] (v) to (y2);
           	\draw[-e+6, thick] (y1) to (y2);
           	\draw[-e+, thick] (y1) to (x1);
           	\draw[thick] (x2) [out=10,in=260] to (v) ;
           	\draw[-e+,thick] (x2) [out=80,in=190] to (v) ;
           	\node[left] at (4.5,0) {\LARGE$\rightarrow$};
       	\end{tikzpicture}
    
    &
       	\begin{tikzpicture}
           	\node[dot] (v+) at (-.5,-.3) [label={[blue]left:$ v^+$}] {};
           	\node[dot] (v-) at (.1,-.3) [label={[blue]right:$ v^-$}] {};
           	\node[dot] (y1+) at (1,1) [label={[blue]left:$ y_1^+$}] {};
           	\node[dot] (y1-) at (1.6,1) [label={[xshift=.45cm, yshift=-.575cm,blue]:$ y_1^-$}] {};
           	\node[dot] (y2+) at (1,-1.6) [label={[xshift=-.2cm, yshift=-.775cm,blue]:$ y_2^+$}] {};
           	\node[dot] (y2-) at (1.6,-1.6) [label={[xshift=.2cm, yshift=-.775cm,blue]:$ y_2^-$}] {};
           	\node[dot] (x1+) at (-2,1) [label={[xshift=-.2cm,blue]:$ x_1^+$}] {};
           	\node[dot] (x1-) at (-1.4,1) [label={[xshift=.2cm,blue]:$ x_1^-$}] {};
           	\node[dot] (x2+) at (-2,-1.6) [label={[xshift=-.2cm, yshift=-.775cm,blue]:$ x_2^+$}] {};
           	\node[dot] (x2-) at (-1.4,-1.6) [label={[xshift=.2cm, yshift=-.775cm,blue]:$ x_2^-$}] {};
           	\node[dot] (y) at (3,-.3) [fill=red,label={[red]right:$y$}] {};
           	\node[dot] (x) at (-3,0-.3) [fill=red,label={[red]left:$x$}] {};
           		\node[above] at (0,1.7) {\large$\hat B$};
           	\draw[-e+, thick] (y1-) [out=250,in=45] to (v+) ;
           	\draw[-e+, thick] (v-) [out=290,in=140] to (y2+) ;
           	\draw[-e+4, thick] (y1-) [out=270,in=70] to (y2+);
           	\draw[-e+6, thick] (y1-) [out=220, in=330] to (x1+);
           	\draw[red,thick] (y) to (y2-) ;
           	\draw[red,thick] (y) [out=80,in=50] to (y1+) ;
           	\draw[-e+4,red,thick] (x) to (x2+) ;
           	\draw[-e+4,red,thick] (x) to (x1+) ;
           	\draw[thick] (x2-) [out=30,in=240] to (v-) ;
           	\draw[-e+, thick] (x2-) to (v+);
                \foreach \x/\y in {v+/v-,x1+/x1-,x2+/x2-,y1+/y1-,y2+/y2-} \draw[-e+,thick,color=blue] (\x) to (\y); 	
            \end{tikzpicture}
            \end{tabular}   
	\caption{The construction of $\hat{B}$ from the bidirected graph $B$ with specified sets \textcolor{red}{$\X$} and \textcolor{red}{$\Y$} of signed vertices as in \cref{sec:vertex-disjoint}. \label{fig:Construction_of_auxiliary_graph}}
\end{figure}

Now we turn our attention to the relation between paths in $B$ and paths in $\hat{B}$.
Let~$P = v_1 \ve_1 v_2 \dots v_{n-1} \ve_{n-1} v_{n}$ be any~$\X$--$\Y$~path in~$B$, and write~$\alpha_\ell := \sigma(v_\ell,e_\ell)$ and $\beta_\ell := \sigma(v_{\ell+1}, e_\ell)$ for all~$\ell < n$.
Since~$P$ is an~$\X$--$\Y$ path in~$B$, we have~$\beta_\ell = -\alpha_{\ell+1}$ for all~$\ell < n$ as well as~$(v_1, \alpha_1) \in \X$ and~$(v_n, \beta_{n-1}) \in \Y$.
Thus, the construction of~$\hat{B}$ guarantees that
\begin{align*}
    \hat{P} :=
        &x (x,v_1^{-\alpha_1}) v_1^{-\alpha_1} (v_1^{-\alpha_1}, v_1^{\alpha_1}) v_1^{\alpha_1} (v_1^{\alpha_1},v_{2}^{\beta_1}) v_2^{\beta_1} \dots \\
        &\dots v_{n-1}^{\alpha_{n-1}} (v_{n-1}^{\alpha_{n-1}},v_{n}^{\beta_{n-1}}) v_n^{\beta_{n-1}} (v_n^{\beta_{n-1}}, v_n^{-\beta_{n-1}}) v_n^{-\beta_{n-1}} (v_{n}^{-\beta_{n-1}},y) y
\end{align*}
is an~$x$--$y$~path in~$\hat{B}$.
Moreover, if $\X$--$\Y$ paths~$P_1, \dots, P_k$ are vertex-disjoint, then the corresponding~$x$--$y$ paths~$\hat{P}_1, \dots, \hat{P}_k$ in~$\hat{B}$ are internally vertex-disjoint and in particular edge-disjoint.

With this transfer of~$\X$--$\Y$~paths in~$B$ to~$x$--$y$~paths in~$\hat{B}$ at hand, it is easy to show that an edge-separator of~$x$ and~$y$ in~$\hat{B}$ yields a vertex-separator of~$\X$ and~$\Y$ in~$B$ of at most the same size:

\begin{lemma}\label{lem:separator_implies_separator}
    If there is a set $\hat{S} \subseteq E(\hat{B})$ of size at most $k$ such that there is no $x$--$y$ path in $\hat{B} - \hat{S}$, then there is a set $S \subseteq V(B)$ of size at most $k$ such that there is no $\X$--$\Y$~path in~$B - S$.
\end{lemma}
\begin{proof}
    Let $\hat{S}$ be a set of at most $k$ edges in $\hat{B}$ such that there is no $x$--$y$ path in $\hat{B} - \hat{S}$.
    We define a set~$S$ of~$k$ vertices of~$B$ as follows: 
    \begin{itemize}
        \item if~$\hat{e} \in \hat{S}$ has the form $v^\alpha v^{-\alpha}$, then we add $v$ to $S$,
        \item if~$\hat{e} \in \hat{S}$ has the form $v^\alpha w^{\beta}$, then we arbitrarily add one of~$v$ and~$w$ to $S$, and
        \item if~$\hat{e} \in \hat{S}$ has the form $v^\alpha x$ or $v^\alpha y$, then we add $v$ to $S$.
    \end{itemize}
    Then there is no $\X$--$\Y$~path~$P$ in $B - S$, as the corresponding path~$\hat{P}$ in~$\hat{B}$ would clearly avoid~$\hat{S}$.
\end{proof}

We now turn to the transfer of~$x$--$y$~trails in~$\hat{B}$ to~$\X$--$\Y$~paths in~$B$.
For this, let us first observe that such trails admit a very particular structure due to the construction of~$\hat{B}$:

\begin{proposition} \label{obs:vertex_menger_properties_aux_graph}
    Let $T$ be either an $x$--$y$~trail or a nontrivial $x$--$x$~trail in $\hat{B}$. Then
    \begin{enumerate}[label=(\arabic*)]
        \item\label{itm:prop_trails_in_hat_b_1} every internal vertex of $T$ has the form~$v^{\alpha}$ for some~$v \in V(B)$ and~$\alpha \in \{+, -\}$,
        \item\label{itm:prop_trails_in_hat_b_2} every internal vertex of $T$ is met by $T$ exactly once, and
        \item\label{itm:prop_trails_in_hat_b_3} precisely every second edge of $T$ is of the form~$(v^\alpha, v^{-\alpha})$.
    \end{enumerate}
\end{proposition}
\begin{proof}
    All edges incident to~$x$ or~$y$ have sign~$-$ at~$x$ or~$y$, so every internal vertex of a trail in~$\hat{B}$ has the form~$v^{\alpha}$ for some~$v \in V(B)$ and~$\alpha \in \{+, -\}$, which shows \labelcref{itm:prop_trails_in_hat_b_1}. For every vertex of this form~$v^\alpha$ (i.e.\ every vertex of~$\hat{B}$ other than~$x$ and~$y$), there exists precisely one edge in~$\hat{B}$ which has sign~$-\alpha$ at~$v^\alpha$, namely the edge~$v^{\alpha} v^{-\alpha}$. This implies that $T$ uses the edge~$(v^\alpha, v^{-\alpha})$ either immediately before or after meeting $v^\alpha$. In particular, \labelcref{itm:prop_trails_in_hat_b_3} holds. Since $v^\alpha$ is neither the start- nor the endvertex of $T$, it is met by $T$ exactly once, which proves \labelcref{itm:prop_trails_in_hat_b_2}.
\end{proof}

Now let $\hat{T} := v_1 \vehat_1 v_2 \vehat_2 \dots \vehat_{n-2} v_{n-1} \vehat_{n -1} v_n$ be either an $x$--$y$~trail or a nontrivial $x$--$x$~trail in $\hat{B}$.
By \cref{obs:vertex_menger_properties_aux_graph}~\labelcref{itm:prop_trails_in_hat_b_1} and~\labelcref{itm:prop_trails_in_hat_b_2}, the trail $\hat{T}$ meets all vertices of $\hat{B}$ but $x$ at most once.
Thus, its subtrail $\hat{T}':= \vehat_2 \hat{T} \vehat_{n - 2}$ is a path.

Now replace every subtrail of~$\hat{T}'$ that has the form~$v^\alpha (v^\alpha, v^{-\alpha}) v^{-\alpha}$ by the vertex~$v \in B$;
note in particular that both~$v_2 \vehat_2 v_3$ and~$v_{n-2} \vehat_{n-2} v_{n-1}$ have this form.
Moreover, replace any edge of~$\hat{T}'$ that has the form~$(v^\alpha, w^{\beta})$ for distinct~$v, w \in V(B)$ by the oriented edge in~$B$ that has sign~$\alpha$ at~$v$ and sign~$\beta$ at~$w$ (such an edge exists by the construction of~$\hat{B}$).
\cref{obs:vertex_menger_properties_aux_graph}~\labelcref{itm:prop_trails_in_hat_b_3} then guarantees that the constructed~$P$ is indeed a path in~$B$.

We show that~$P$ starts in~$\X$; the case of~$P$ ending in~$\X$ or in $\Y$ is symmetrical.
The only neighbours of~$x$ in~$\hat{B}$ are the vertices~$v^\alpha$ with~$(v, -\alpha) \in \X$.
\cref{obs:vertex_menger_properties_aux_graph}~\labelcref{itm:prop_trails_in_hat_b_3} then implies that if~$v_2 = v^{\alpha}$, then~$\vehat_2 = (v^{\alpha}, v^{-\alpha})$ and so~$\hat{e}_3$ has sign~$-\alpha = -\hat{\sigma}(v_3, \hat{e}_2)$ at~$v_3 = v^{-\alpha}$.
Thus, our construction of~$P$ guarantees that~$P$ starts in~$(v, -\alpha) \in \X$, as desired.

So if~$\hat{T}$ is an~$x$--$y$~trail, then~$P$ starts in~$\X$ and ends in~$\Y$ and it is indeed an~$\X$--$\Y$~path in~$B$, since we assumed~$B$ to satisfy~\labelcref{item:graph_appearance_assumption_X-Y_path_is_immediate_1} in~\cref{prop:SimplifyXYpaths}.
And if~$\hat{T}$ is a nontrivial~$x$--$x$~trail, then~$P$ starts and ends in~$\X$ and it is nontrivial, since~$B$ also satisfies~\labelcref{item:graph_appearance_assumption_X-X_path_is_immediate_2} by assumption.

This transfer of~$x$--$y$~trails from~$\hat{B}$ to~$\X$--$\Y$~paths in~$B$ now interacts with the transfer from~$B$ to~$\hat{B}$ precisely in the desired way:    
\begin{lemma}\label{lem:paths_imply_paths}
    Let~$P_1, \dots, P_k$ be~$\X$--$\Y$~paths in~$B$, and let~$\hat{P}_1, \dots, \hat{P}_k$ be the corresponding \hbox{$x$--$y$~paths} in~$\hat{B}$.
    If there are~$k + 1$ edge-disjoint $x$--$y$~paths $\hat{P}_1', \dots, \hat{P}_{k+1}'$ in~$\hat{B}= (\hat{G}, \hat{\sigma})$ where~$\hat{P}_i'$ starts in the same edge as $\hat{P}_i$ for $i \in [k]$, then there are~$k + 1$ vertex-disjoint $\X$--$\Y$~paths $P_1', \dots, P_{k+1}'$ in~$B$ where~$P_i'$ starts in the same vertex as $P_i$ for $i \in [k]$.
\end{lemma}
\begin{proof}
	Let $P_1', \dots, P_{k +1}'$ be the $\X$--$\Y$~paths in $B$ obtained from $\hat{P}_1', \dots, \hat{P}_{k +1}'$.
	Let $i \in [k]$ be arbitrary, and let $v^\alpha$ be the endvertex of the first edge of~$\hat{P}_i'$.
	By construction, the path $P_i'$ starts in the vertex $v$.
	Now our assumption on $\hat{P}_i'$ implies that the path $\hat{P}_i'$ starts in the same edge as $\hat{P}_i$.
	By the construction of $\hat{P}_i$ from~$P_i$, the endvertex $v^\alpha$ of the first edge of $\hat{P}_i$ has the property that $v$ is the startvertex of $P_i$.
	Thus, $P_i$ and $P_i'$ start in the same vertex, as desired.
	
	It remains to prove that the paths $P_1', \dots, P_{k +1}'$ are vertex-disjoint:
	\cref{obs:vertex_menger_properties_aux_graph}~\labelcref{itm:prop_trails_in_hat_b_3} together with our construction yields that if two distinct~$P_i'$ and~$P_j'$ would contain the same vertex~$v$, then both~$\hat{P}_i'$ and~$\hat{P}_j'$ would have to contain the edge~$v^\alpha v^{-\alpha}$, which contradicts that they are edge-disjoint.
\end{proof}

With~\cref{lem:separator_implies_separator} and~\cref{lem:paths_imply_paths}, we have now shown that the construction of~$\hat{B}$ indeed provides a general framework to deduce a vertex-version of Menger's theorem in bidirected graphs form an edge-version.
We can thus finally deduce~\cref{theo:vertex_menger} from~\cref{theo:edge_menger} by checking that~$x$ is edge-clean in~$\hat{B}$ if~$\X$ is clean in~$B$:
    
\begin{proof}[Proof of~\cref{theo:vertex_menger}]
    By~\cref{prop:SimplifyXYpaths}, we may assume that~$B$, $\X$ and~$\Y$ satisfy~\labelcref{item:graph_appearance_assumption_X-Y_path_is_immediate_1,item:graph_appearance_assumption_X-X_path_is_immediate_2}, and this does not affect the cleanness of~$\X$ in~$B$.
    Let~$\hat{B}$, $x$ and~$y$ be constructed from~$B$, $\X$ and~$\Y$ as above, and suppose for a contradiction that~$x$ is not edge-clean, i.e.\,that there exists a nontrivial~$x$--$x$~trail~$\hat{T}$ in~$\hat{B}$.
    As shown above, the trail~$\hat{T}$ then gives rise to a nontrivial path in $B$ that starts and ends in $\X$, which contradicts the cleanness of~$\X$ in~$B$.
    Thus, we can apply \cref{theo:edge_menger} to $\hat{B}$, $x$, $y$ and the~$k$ edge-disjoint~$x$--$y$~paths $\hat{P}_1, \dots, \hat{P}_k$ in~$\hat{B}$ corresponding to the~$k$ vertex-disjoint~$\X$--$\Y$~paths~$P_1, \dots, P_k$ in~$B$.
    Depending on the outcome of~\cref{theo:edge_menger}, \cref{lem:separator_implies_separator} or~\cref{lem:paths_imply_paths} then complete the proof.
\end{proof}

We remark that it is easy to deduce a version of~\cref{theo:vertex_menger} in which an~$\X$--$\Y$~path does not have to be internally disjoint from~$V(\X) \cup V(\Y)$ and may contain trivial $\X$--$\Y$~paths:
Apply~\cref{theo:vertex_menger} to the bidirected graph that is obtained by adding a new vertex for every~$v \in V(\X)$ which is joined to~$v$ by an edge with sign~$-\alpha$ at~$v$ and sign~$\alpha$ at the newly added vertex for any $(v, \alpha) \in \X$, and by adding vertices and edges for $v \in V(\Y)$ in the same way. Then take~$\X'$ and~$\Y'$ as the respective sets of newly added vertices.

\cref{theo:vertex_menger} implies Menger's Theorem for vertex-disjoint paths in both graphs and directed graphs.
As explained at the end of \cref{sec:edges-disjoint}, we may regard any undirected or directed graph as a bidirected graph.
We then obtain the desired version of Menger's Theorem for vertex sets~$X$ and~$Y$ in a directed graph~$D$ by considering the sets~$\X := \set{(x,-) \mid x \in X}$ and~$\Y = \{(y, +) \mid y \in Y\}$ of signed vertices.
Note that~$\X$ is indeed clean since every path in the directed graph~$D$ which ends in~$X$ does so in~$(x, +)$ for some~$x \in X$ and hence not in~$\X$.

\section{Polynomial time algorithm}\label{sec:polynomial_time_algorithm}

In this \namecref{sec:polynomial_time_algorithm} we show that the proof of the edge-version of Menger's Theorem for bidirected graphs, \cref{theo:edge_menger}, can be turned into a polynomial time algorithm.
More precisely, we show the following statement:

\begin{theorem} \label{theo:algorithm_edge_menger}
    There exists a polynomial time algorithm that, given distinct vertices~$x$ and~$y$ of a bidirected graph~$B$ and~$k$ edge-disjoint~$x$--$y$~paths~$P_1, \dots, P_k$ in~$B$ where~$P_i$ starts in~$\ve_i$ for~$i \in [k]$, finds either~$k+1$ edge-disjoint~$x$--$y$~paths~$P_1', \dots, P_{k+1}'$ in~$B$ where~$P_i'$ starts in~$\ve_i$ for~$i \in [k]$ or a set~$S$ of~$k$ edges of~$B$ such that~$B - S$ does not contain any~$x$--$y$~path.
\end{theorem}

As a consequence, there exists a polynomial time algorithm for \cref{theo:vertex_menger} since the transfer process introduced in \cref{sec:vertex-disjoint} can be computed in polynomial time.

For the proof of \cref{theo:algorithm_edge_menger} note that the proof of~\cref{theo:edge_menger} uses a recursive procedure by induction on the sum of the lengths of the paths~$P_i$; it is thus enough to show that each recursion step runs in polynomial time as the number of recursion steps is bounded by~$|E(B)|$.

To show that each recursion step runs in polynomial time, we will in particular need polynomial time algorithms to find paths and trails between two fixed signed vertices as well as the appendage of a given path.
One key ingredient to this is Edmonds' celebrated Blossom Algorithm:

\begin{theorem}[\cite{edmonds1965paths}] \label{theo:blossom_algo}
    There exists a polynomial time algorithm which, given a graph~$G$, either finds a perfect matching in~$G$ or determines that~$G$ has no perfect matching.
\end{theorem}

We first show that a path between two signed vertices can be found in polynomial time.

\begin{lemma}\label{lem:algorithm_path}
    There exists a polynomial time algorithm that, given a bidirected graph~$B$ and signed vertices~$(x, \alpha)$ and~$(y, \beta)$ of~$B$ with $x \neq y$, either finds an $(x, \alpha)$--$(y, \beta)$~path in $B$ or determines that no such path exists.
\end{lemma}
\begin{proof}
    We first describe an algorithm which runs in polynomial time and then prove its correctness.
    
    The algorithm starts by checking if there exists an edge in~$B$ with sign~$\alpha$ at~$x$ and sign~$\beta$ at~$y$, which then forms our desired path.
    If no such edge exists, then the algorithm considers the graph~$B' := B - \{x, y\}$ and computes the set~$\X'$ of all signed vertices~$(v, \gamma)$ of~$B'$ for which there exists an $(x, \alpha)$--$(v,\gamma)$~edge and the set~$\Y'$ which is defined analogously with respect to~$y$.
    As in \cref{sec:vertex-disjoint}, the algorithm then constructs the auxiliary graph~$\hat{B}' = (\hat{G}', \hat{\sigma}')$ from~$B'$,~$\X'$ and~$\Y'$; let~$x'$ and~$y'$ be the two vertices of~$\hat{B}'$ representing~$\X'$ and~$\Y'$, respectively.
    Note that the construction of~$\hat{B}'$ can be done in polynomial time.

    The algorithm then uses the polynomial time algorithm of~\cref{theo:blossom_algo} to construct a perfect matching of~$\hat{G}'$ if possible.
    If there is a perfect matching, then the algorithm computes an~$(x, \alpha)$--$(y, \beta)$~path in~$B$ from it in polynomial time (see below).
    Otherwise, the algorithm returns that there is no $(x, \alpha)$--$(y, \beta)$~path in $B$.
    This concludes our description of the polynomial time algorithm.

    For the proof of its correctness, let~$\hat{e}_v$ denote the unique edge of~$\hat{G}'$ between~$v^+$ and~$v^-$ for every vertex~$v \in V(B')$.
    Then~$M := \{\hat{e}_v \mid v \in V(B')\}$ is a matching in~$\hat{G}'$, and every vertex of~$\hat{G}'$ except from~$x'$ and~$y'$ is incident to an edge in~$M$.
	
	If there is a perfect matching~$M'$ in~$\hat{G}'$, then let~$H'$ be the subgraph of~$\hat{G}'$ with edge set~$M \cup M'$.
    By the definition of~$M$, the component of~$H'$ containing~$x'$ is an~$x'$--$y'$ path~$\hat{P}$.
    The construction of~$\hat{B}'$ together with the definition of~$M$ implies that such a path~$\hat{P}$ in~$\hat{G}'$ induces an~$x'$--$y'$~path~$\hat{P}'$ in~$\hat{B}'$.
    As in the proof of~\cref{theo:vertex_menger}, this path~$\hat{P}'$ then induces an~$\X'$--$\Y'$~path~$P'$ in~$B'$.
    The definition of~$\X'$ and~$\Y'$ then allows us to extend the path~$P'$ to an~$(x, \alpha)$--$(y, \beta)$~path in~$B$.
	
	Conversely, we show that if there exists an~$(x, \alpha)$--$(y, \beta)$~path in~$B$, then there is a perfect matching in~$\hat{G}$.
    Let~$P = x \ve_1 v_1 \dots v_{\ell-1} \ve_\ell y$ be an~$(x, \alpha)$--$(y, \beta)$~path in~$B$.
    Then~$P' = v_1 \dots v_{\ell-1}$ is an~$\X'$--$\Y'$ path in~$B'$ by the definition of~$\X'$ and~$\Y'$.
    As in the proof of~\cref{theo:vertex_menger}, the path~$P'$ induces an~$x'$--$y'$~path~$\hat{P}'$ in~$\hat{B}'$.
    The respective path~$\hat{Q}'$ in~$\hat{G}'$ alternates between edges in~$M$ and not in~$M$.
    Note that all the vertices of $\hat{G}'$ but the startvertex and endvertex of $\hat{Q}'$ are incident to edges in $M$.
    This means that no vertex of $\hat{Q}'$ is incident to an edge in $M \setminus E(\hat{Q}')$. 
    It follows that the symmetric difference~$(M \setminus E(\hat{Q}')) \cup (E(\hat{Q}') \setminus M)$ is a perfect matching in~$\hat{G}'$, as desired.
\end{proof}

We next extend~\cref{lem:algorithm_path} from paths to trails.
Our main tool for this is the line graph of a bidirected graph which we introduce below.
For this we need one more definition:
An \emph{orientation of a bidirected graph~$B$} is a map~$\nu: E(B) \to \vecev{E(B)}$ such that~$\nu(e) \in \{ \ve, \ev \}$ for any~$e \in E(B)$.

\begin{definition}
    Given a bidirected graph $B = (G, \sigma)$ and an orientation $\nu$ of $B$, we define the \emph{line graph of~$B$ with respect to $\nu$}, denoted as~$L := L_\nu = (H, \tau)$, as follows.
    The graph~$H$ has vertex set~$E(G)$ and contains an edge~$\{e_1, e_2\} \in E(H)$ with label~$v$ for any two distinct edges~$e_1, e_2 \in E(G)$ and a vertex~$v \in e_1 \cap e_2$ satisfying~$\sigma(v, e_1) = - \sigma(v, e_2)$.
    The signing~$\tau$ of the half-edges~$\halfedges(H)$ assigns to an edge~$\{e_1, e_2\} \in E(H)$ with label~$v$ the sign
    \begin{equation*}
        \tau(e_i, \{ e_1, e_2 \} ) :=
        \begin{cases}
        + & \text{if } v \text{ is endvertex of } \nu(e_i), \\
        - & \text{if } v \text{ is startvertex of } \nu(e_i),
        \end{cases}
    \end{equation*}
    for~$i = 1, 2$.
\end{definition}

Fix an arbitrary orientation of a bidirected graph~$B$, and let~$L$ be the corresponding line graph of~$B$.
Just as for line graphs of (un)directed graphs, the above definition of~$L$ implies that walks in~$B$ induce walks in~$L$ and vice versa.
Indeed, a walk~$v_0 \ve_1 v_1 \dots v_{n-1} \ve_n v_n$ in~$B$ induces the walk~$e_1 \va_1 e_2 \dots e_{n-1} \va_{n-1} e_n$ in~$L$ where~$\va_i$ is the directed edge from~$e_i$ to~$e_{i+1}$ with label~$v_i$ for any~$i \in [n-1]$.
Vice versa, consider a walk~$W = e_1 \va_1 e_2 \dots e_{n-1} \va_{n-1} e_n$ in~$L$ which has label~$v_i$ at the edge~$a_i$ for any~$i \in [n-1]$, and let~$v_0 \in e_1 \setminus \{v_1\}$ and~$v_n \in e_n \setminus \{v_{n-1}\}$.
Writing~$\ve_i = (e_i, v_{i-1}, v_i)$ for~$1 \le i \le n$, the walk~$W$ in~$L$ then induces the walk~$v_0 \ve_1 v_1 \dots \ve_n v_n$ in~$B$.

The above correspondences between walks in~$B$ and walks in~$L$ in particular imply that trails in~$B$ induce paths in~$L$ and paths in~$L$ induce trails in~$B$.
We can thus compute trails in~$B$ in polynomial time by applying \cref{lem:algorithm_path} to a suitable line graph of~$B$.

\begin{corollary} \label{cor:algorithm_trail}
    There exists a polynomial time algorithm that, given a bidirected graph~$B$ and signed vertices~$(x, \alpha)$ and~$(y, \beta)$ of~$B$ with~$x \neq y$, either finds an~$(x, \alpha)$--$(y, \beta)$~trail in~$B$ or determines that no such trail exists.
\end{corollary}
\begin{proof}
    Choose an orientation~$\nu$ of~$B$ such that the startvertex of~$\nu(e)$ is~$x$ for all~$e \in E(B)$ incident to~$x$ and such that the endvertex of~$\nu(f)$ is~$y$ for all~$f \in E(B)$ incident to~$y$.
    Let~$L$ be the line graph of~$B$ with respect to~$\nu$.
    Then the result follows by applying~\cref{lem:algorithm_path} to every pair~$(e, +)$ and~$(f, -)$ of signed vertices of~$L$ such that in~$B$ the edge~$e$ has sign~$\alpha$ at~$x$ and the edge~$f$ has sign~$\beta$ at~$y$; the choice of~$\nu$ here implies that any such~$(e, +)$--$(f, -)$~path in~$L$ indeed translates to an~$x$--$y$~trail in~$B$ starting in~$x$ with sign~$\alpha$ and ending in~$y$ with sign~$\beta$.
\end{proof}

As our polynomial time algorithm for~\cref{theo:algorithm_edge_menger} follows the proof of~\cref{theo:edge_menger}, we will also need to compute the appendage of a path in polynomial time.
For this, we first prove a lemma which reduces this problem to finding `ear trails' in polynomial time which are defined as follows.

Let~$P$ be a path in a bidirected graph~$B$, and let~$v$ be the startvertex of~$P$.
Let~$A$ be a~$(P, v)$-admissible set.
The set~$\Bon(P, v, A)$ is then defined to contain~$(v, +)$ and~$(v, -)$ as well as all signed vertices~$(w, \alpha)$ of~$B$ for which there exists an oriented edge~$\ve \in \vAv \cup \vE(P)$ that points at $w$ with sign~$-\alpha$.
Note that for each~$(w, \alpha) \in \Bon(P, v, A)$, there is a~$v$--$(w, -\alpha)$~trail in~$A \cup E(P)$.
A nontrivial trail~$T$ in~$B$ is a \emph{$(P, v, A)$-quasi-ear trail} if both the signed start- and endvertex of~$T$ are contained in~$\Bon(P, v, A)$ and the first and last edge of~$T$ are not in~$\vAv \cup \vEv(P)$.
If this latter condition holds additionally for all edges of~$T$, that is~$E(T) \cap (A \cup E(P)) = \emptyset$, then~$T$ is an~\emph{$(P, v, A)$-ear trail}.

We observe that for every~$(P, v, A)$-ear trail~$T$ in~$B$, the set~$A \cup E(T)$ is~$(P, v)$-admissible.
Indeed, the signed endvertex~$(w, \alpha)$ of~$T$ is in~$\Bon(P, v, A)$ and thus there exists a~$v$--$(w, -\alpha)$~trail $T_w$ in~$A \cup E(P)$.
So for every~$\ve \in \vE(T)$ the concatenation of~$\ve T$ and $\inv{T}_w$ is a~$\ve$--$v$~trail in~$A \cup E(P) \cup E(T)$.
Similarly, there exists an~$\ve$--$v$~trail in~$A \cup E(P) \cup E(T)$ for every~$\ve \in \vE(\inv{T})$.

\begin{lemma} \label{lem:construction_appendage}
	Let~$P$ be a path in a bidirected graph~$B$, and let~$v$ be the startvertex of~$P$.
	For every non-maximal~$(P, v)$-admissible set~$A \subsetneq A(P, v)$, there is either an edge~$e \in E({P}) \setminus A$ such that~$A \cup \{e\}$ is~$(P, v)$-admissible or there exists an~$(P, v, A)$-ear trail.
\end{lemma}
\begin{proof}
    Assume that there is no edge~$e \in E(P) \setminus A$ such that~$A \cup \{e\}$ is~$(P,v)$-admissible.
    This implies the following observation which we will use repeatedly throughout this proof:
    Let~$Q$ be a~$v$--$\ve$~trail in~$A(P, v) \cup E(P)$ for some orientation $\ve$ of $e$.
    Then~$\ve \notin \vE(\inv{P}) \setminus \vAv$ as otherwise~$A \cup \{e\}$ is~$(P, v)$-admissible witnessed by~$Q \ve$ and~$P \ev$, a contradiction to the assumption.

    The proof now proceeds in two steps.
    We first find a~$(P, v, A)$-quasi-ear trail in~$B$.
    Secondly, we show that every~$(P, v, A)$-quasi-ear trail contains a~$(P, v, A)$-ear trail as a subtrail.

    Since~$A$ is a non-maximal~$(P, v)$-admissible set, there exists~$e \in A(P, v) \setminus A$.   
    The above observation then implies that no orientation~$\ve$ of~$e$ can be in~$\vE(\inv{P})$, since~$e \in A(P, v)$ implies the existence of a~$v$--$\ve$~trail in~$A(P, v) \cup E(P)$; so we have~$e \in A(P, v) \setminus (A \cup E(P))$.

    Fix an arbitrary orientation~$\ve$ of~$e$.
    Since~$e \in A(P, v)$, there is a~$v$--$\ve$~trail~$T$ in~$A(P, v) \cup E(P)$, and we let~$\vf$ be the first edge on~$T$ not contained in~$\vAv \cup \vE(P)$; note that such an edge~$\vf$ exists since~$\ve$ is a suitable choice.
    The above observation then implies~$\vf \notin \vE(\inv{P})$, and we hence have~$f \in A(P, v) \setminus (A \cup E(P))$.
    Thus, there exists a~$v$--$\fv$~trail~$T'$ in~$A(P,v) \cup E(P)$ and, as above, we let~$\vf'$ be the first edge on~$T'$ not contained in~$\vAv \cup \vE(P)$.
    Again, we have~$\vf' \in A(P, v) \setminus (A \cup E(P))$ by the above observation.
    By this construction,~$\vf' T' \fv$ is a~$(P, v, A)$-quasi-ear trail in~$B$.
    
    Now let~$T$ be a~$(P, v, A)$-quasi-ear trail in~$B$.
    By definition,~$T$ is a~$(P, v, A)$-ear trail if and only if~$E(T) \cap (A \cup E(P)) = \emptyset$.
    So suppose that~$T$ is not itself a~$(P, v, A)$-ear trail, and let~$\vf T \vf'$ be a maximal subtrail of~$T$ in~$A \cup E(P)$. 
    Since the first and the last edge of~$T$ are not in~$\vAv \cup \vEv(P)$, there exist an edge~$\vg$ preceding~$\vf$ on~$T$ and an edge~$\vg'$ succeeding~$\vf'$ on~$T$.
    We now claim that at least one of~$T \vg$ and~$\vg' T$ is again a~$(P, v, A)$-quasi-ear trail.
    This then concludes the proof:
    Since both~$T \vg$ and~$T \vg'$ contain strictly fewer edges in~$A \cup E(P)$, we can recursively apply the above step until we obtain a~$(P, v, A)$--ear trail.

    Both~$\vg$ and~$\vg'$ are not in~$\vAv \cup \vEv(P)$ by construction.
    So by the definition of~$\Bon(P, v, A)$, the claim holds if~$\vf \in \vAv \cup \vE(\inv{P})$ or~$\vf' \in \vAv \cup \vE(P)$.
    Thus, we can assume that $\vf \in \vE(P) \setminus \vAv$ and~$\vf' \in \vE(\inv{P}) \setminus \vAv$ since $f, f' \in A \cup E(P)$. Let $\vh$ be the first edge of $\vEv(\vf T \vf')$ along $P$. If we have~$\vh \in \vE(T)$, then $P \vh T \vf'$ is a trail.
    This contradicts our above observation since $\vf' \in \vE(\inv{P}) \setminus \vAv$. Otherwise, $\vh \in \vE(\inv{T})$ and then $P \vh \inv{T} \fv$ is a trail.
    This contradicts our above observation since $\fv \in \vE(\inv{P}) \setminus \vAv$.
\end{proof}

\begin{corollary} \label{cor:algorithm_appendage}
    There exists a polynomial time algorithm that, given a path~$P$ with startvertex~$v$ in a bidirected graph~$B$, computes its appendage~$A(P, v)$.
\end{corollary}
\begin{proof}
    The algorithm starts with the~$(P, v)$-admissible set~$A = \emptyset$ and extends~$A$ recursively to~$(P, v)$-admissible sets until~$A = A(P, v)$.
    Let~$A$ be any~$(P, v)$-admissible set.
    
    The algorithm first checks whether there exists an edge~$e \in E(P) \setminus A$ such that~$A \cup \{e\}$ is again~$(P, v)$-admissible.
    Such an edge~$e$ exists if and only if there exists a~$v$--$\ve~$trail in~$A \cup E(P)$ where~$\ve \in \vE(\inv{P})$.
    This can be checked in polynomial time by~\cref{cor:algorithm_trail}.
    If such an edge~$e$ exists, then the algorithm restarts the recursion with~$A \cup \{e\}$.

    If there is no such edge, then the algorithm looks for a~$(P, v, A)$-ear trail~$T$ in~$B$.
    This can again be done in polynomial time by~\cref{cor:algorithm_trail} since there exists such a~$(P, v, A)$-ear trail if and only if there is a nontrivial trail in~$B - (A \cup E(P))$ that starts and ends in~$\Bon(P, v, A)$.
    If such~$T$ exists, then the algorithm restarts the recursion with~$A \cup E(T)$ which is again~$(P, v)$-admissible by the definition of~$(P, v, A)$-ear trail.

    If there is also no~$(P, v, A)$-ear trail, then~\cref{lem:construction_appendage} implies~$A = A(P, v)$, and the algorithm terminates.
    Since the number of recursion steps is bounded by~$|E(B)|$, this algorithm runs in polynomial time.
\end{proof}

With all these polynomial time algorithms at hand, we are now ready to prove~\cref{theo:algorithm_edge_menger}. 

\begin{proof}[Proof of~\cref{theo:algorithm_edge_menger}]
    Our algorithm follows the construction in the proof of~\cref{theo:edge_menger} which thus yields the correctness of the algorithm.
    Each construction step of the recursive procedure in the proof of~\cref{theo:edge_menger} can be done in polynomial time by~\cref{lem:algorithm_path}, \cref{cor:algorithm_trail} and \cref{cor:algorithm_appendage}.
    The number of recursion steps is at most the sum of the lengths of the paths~$P_i$ and hence bounded by~$|E(B)|$.
    Altogether, this yields a polynomial running time.
\end{proof}

\section*{Acknowledgements}

The second author carried out this work during a Humboldt Research Fellowship at the University
of Hamburg. He thanks the Alexander von Humboldt-Stiftung for financial support.

The fourth author gratefully acknowledges support by doctoral scholarships of the Studienstiftung des deutschen Volkes and the Cusanuswerk -- Bisch\"{o}fliche Studienf\"{o}rderung.

\bibliography{reference}

\end{document}